\documentclass[reqno]{amsart}
\usepackage[utf8]{inputenc}
\usepackage{amssymb, amsmath, amsthm, amsfonts, xstring, graphicx, enumerate, mathtools, mathrsfs}
\usepackage[pdftex]{hyperref}
\usepackage{cite}

\theoremstyle{plain}
\newtheorem{theorem}    {Theorem}[section] 
\newtheorem*{theorem*}    {Theorem}
\newtheorem{lemma}      [theorem]{Lemma}

\theoremstyle{definition}

\theoremstyle{remark}
\newtheorem{remark}              {Remark}

\newtheorem*{question*}   {Question}

\usepackage{thmtools,enumitem}

\usepackage[capitalize]{cleveref}

\declaretheorem[style=plain,name=Theorem,numberwithin=section]{thm}
\declaretheorem[style=plain,name=Lemma,numberlike=thm]{lem}
\declaretheorem[style=plain,name=Proposition,numberlike=thm]{prop}

\declaretheorem[style=plain,name=Corollary,numberlike=thm]{cor}
\declaretheorem[style=plain,name=Definition,numberlike=thm]{defn}

\newlist{thmlist}{enumerate}{1}
\setlist[thmlist]{label=(\roman{thmlisti}),
                  ref=\thethm(\roman{thmlisti}),
                  noitemsep}

\Crefname{thm}{Theorem}{Theorems}
\Crefname{lem}{Lemma}{Lemmas}
\Crefname{listthm}{Theorem}{Theorems}
\Crefname{listlem}{Lemma}{Lemmas}

\addtotheorempostheadhook[thm]{\crefalias{thmlisti}{listthm}}
\addtotheorempostheadhook[lem]{\crefalias{thmlisti}{listlem}}

\newcommand{\F}{\mathbb{F}}
\newcommand{\C}{\mathbb{C}}
\newcommand{\R}{\mathbb{R}}
\newcommand{\Q}{\mathbb{Q}}
\newcommand{\Z}{\mathbb{Z}}
\newcommand{\N}{\mathbb{N}}
\newcommand{\ep}{\varepsilon}
\newcommand{\spanner}[1]{\begingroup\allowdisplaybreaks#1\endgroup}
	\newcommand{\A}{\mathfrak{A}}
	\newcommand{\pitt}[2]{\pi_{E_1\la(#1\ra),E_2\la(#2\ra)}}

	\newcommand{\pit}[1]{\pi_{E\la(#1\ra)}}
	
	\newcommand{\I}{\mathcal I}
	\renewcommand{\H}[2]{H(#1^2-4#2)}
	\renewcommand{\S}{\mathcal S}
	\newcommand{\jp}[1]{J_{#1,p}}
	\renewcommand{\P}{P_{\mathcal F}}
		\newcommand{\cP}{\mathcal P}
	\newcommand{\cF}{\mathcal F}
	\newcommand{\U}{\mathcal U}
	\newcommand{\V}{\mathcal V}
	\newcommand{\congval}{\upsilon}
	\newcommand{\modval}{\omega}
\newcommand{\pw}[1]{\begin{cases}#1\end{cases}}
\newcommand{\abs}[1]{\left\vert#1\right\vert}
\newcommand{\floor}[1]{\la\lfloor #1\ra\rfloor}
\newcommand{\la}{\left}
\newcommand{\ra}{\right}
\newcommand{\set}[1]{\left\{ #1 \right\}}
\renewcommand{\bar}{\overline}
\newcommand{\sm}{\smallsetminus}
\newcommand{\Spec}{\mathrm{Spec}\,}
\newcommand{\norm}[1]{\abs{\!\abs{#1}\!}}
\newcommand{\differ}[1]{\mathrm{d}#1}
\newcommand{\oth}{\text{otherwise}}
\newcommand{\kro}[2]{\la(\frac{#1}{#2}\ra)}
\newcommand{\ceil}[1]{\la\lceil #1\ra\rceil}
\newcommand{\inv}{^{-1}}

\title[Distribution of traces of Frobenius for families of elliptic curves]{On the distribution of traces of Frobenius for families of elliptic curves and the Lang--Trotter conjecture on average}
\author{Nathan Fugleberg}
\author{Nahid Walji}
\address{Department of Mathematics, University of British Columbia, Vancouver, B.C., V6T 1Z2, Canada}
 \email{fugleberg.nathan@gmail.com}
 \email{nwalji@math.ubc.ca}

\date{}

\makeatletter
\@namedef{subjclassname@2020}{%
  \textup{2020} Mathematics Subject Classification}
\makeatother

\begin{document}

\maketitle
\begin{abstract}
    We obtain distribution results for traces of Frobenius for various families of elliptic curves with respect to the Lang--Trotter conjecture, extremal primes, and the central limit theorem.
    This includes some generalisations and bounds related to the work of Sha--Shparlinski on the average Lang-Trotter conjecture for single-parametric families of elliptic curves and the work of various authors on the trace of Frobenius for primes in congruence classes.
    Some results are also obtained for modular forms.
    
\end{abstract}

\section{Introduction}

\subsection{Set-up and earlier results}
Let $E$ be an elliptic curve over $\Q$. For a prime $p$ not dividing the conductor, denote by $a_p (E) := p + 1 - |E(\F_p)|$ the corresponding trace of Frobenius. The distribution of $a_p (E)$ has been greatly studied over many years. One question of interest is to fix a particular value $r$ for the trace and ask about the occurrence of primes for which $a_p (E) = r$. In the case of $r = 0$, any such $p \geq 5$ is called supersingular. Let us assume that $E$ does not have complex multiplication (CM). Then the Lang--Trotter conjecture \cite{lt} predicts that 
\begin{align*}
\pi_E(r,x):=\# \{p \leq x \mid a_p (E) = r\} \sim C_{E,r}\frac{\sqrt{x}}{\log x},
\end{align*}
where $C_{E,r}$ is an explicitly defined constant (depending on the elliptic curve and the trace), which was then refined in \cite{bj}.

Serre \cite{ser2} obtained asymptotic upper bounds on $\pi_E (r,x)$, implying density zero of such primes. Elkies \cite{e} later proved that every rational elliptic curve has infinitely many supersingular primes. The best known upper bound on the occurrence of supersingular primes is $\pi_E (0,x) \ll x^{3/4}$ \cite{e3}. In the case of non-zero traces, we have $\pi_E (r,x) \ll x (\log \log  x)^2(\log x)^{-2}$ \cite{tz}.

In \cite{fm}, Fouvry--Murty showed that the Lang--Trotter conjecture for $r =0$ holds on average for a certain family of elliptic curves, which was extended by David--Pappalardi \cite{dp} (and later improved in \cite{b}) to general $r$: Let $E(a,b)$ denote the elliptic curve with equation $y^2 = x ^3 + ax + b$, for (suitable) $a,b \in \Z$. Then 
\begin{align}
\frac{1}{4AB}\sum_{\substack{|a|\leq A\\ |b|\leq B}} \pi_{E(a,b)} (r,x)
\sim C_r\frac{x^{1/2}}{\log x},\label{avg:dp}
\end{align}
under the condition that $A$ and $B$ have sufficiently fast growth with respect to $x$.

One can also work with elliptic curves over a number field and count, for example, inert or totally split primes associated to a given value for the trace of Frobenius. Similar conjectures and averaging results for these extensions are derived in \cite{cfjkp,dp2,js,w}. Many of these average results rely on \cite{bir} which relates Hurwitz numbers $H(r ^2 -4p)$ to the number of elliptic curves $E/\F_p$ with Frobenius trace $r$.

For a pair of non-CM elliptic curves $(E_1, E_2)$ and two integers $r_1$ and $r_2$ one can extend the heuristics of the Lang--Trotter conjecture to obtain the prediction 
\begin{align*}
\pi_{E_1,E_2}(r_1,r_2,x) := \# \{p \leq x \mid a_p (E_1) = r_1, a_p (E_2) = r_2\}\sim C_{E_1,E_2,r_1,r_2}\log \log  x,
\end{align*}
where $C_{E_1,E_2,r_1,r_2}$ is an explicit constant. A number of averaging results have been obtained, first by Fouvry--Murty \cite{fm2} in the supersingular case and then for any trace, at various levels of generality \cite{adj,ap,dks}.

A question of interest is to obtain average results for other families of elliptic curves. For example, the subset of elliptic curves with non-trivial torsion is too small to affect the averaging results above, and so are of interest to be studied separately. This has been done in \cite{bbij,j}. Averaging results have also been obtained for families of elliptic curves where the coefficients $a$ and $b$ are determined by certain polynomials \cite{ch,cs,s,ss,w11}. 
In this paper, we will derive improvements over bounds given in \cite{ss} which would be optimal (up to a constant) under the Lang-Trotter conjecture.

We will also consider the distribution of extremal primes, which  are the primes $p$ at which $a_p (E) = \pm \left \lfloor 2\sqrt{p} \right \rfloor$. These can be seen as a special case of allowing the values of Frobenius traces to vary according to a specified sequence (as seen in \cite{ss}).
		In \cite{jp}, James--Pollack prove that for elliptic curves $E/\Q$ with CM we have $\pi_E(\left( \pm \left \lfloor 2\sqrt{p} \right \rfloor \right),x) \sim \frac{2}{3\pi}\frac{x^{3/4}}{\log x}$ (where in our notation we have replaced the constant $r$ with a prime-indexed sequence).
In \cite{jttwz} they conjecture that for non-CM $E/\Q$ there exists a constant $C_E$ such that
		\begin{align*}
			\pi_E(\left( \pm \left \lfloor 2\sqrt{p} \right \rfloor \right),x)	&\sim C_E\frac{x^{1/4}}{\log x}.
		\end{align*}
		In \cite{gj} they show that this holds on average.
Here, we will obtain a number of upper bounds relating to the distribution of extremal primes on average, for various families of elliptic curves.

Our final question of interest on the distribution of traces of Frobenius concerns central limit theorems. In \cite{MP20}, Murty--Prabhu obtain a central limit theorem that is an elliptic curve analogue of Theorem 2 of Nagoshi \cite{Na06} for modular forms (see also \cite{CK17, WX17,PS19, BPS20} for central limit theorems in related settings). Let $h$ be a continuous function on $\R$ and denote by $\widetilde{a}_{E}(p) := a_E (p) / \sqrt{p}$ the normalised trace of Frobenius for $E$ at $p$. Then, under the condition that $A,B$ have asymptotic growth such that $\log A / \log x$ and $\log B / \log x \rightarrow \infty$ as $x \rightarrow \infty$, they show
\begin{align*}
\sum_{\substack{|a| \leq A,|b| \leq B \\ \Delta (a,b)\neq 0}} h \left(\frac{\sum_{p \leq x} \widetilde{a}_{E(a,b)}(p)}{\sqrt{\pi(x)}}\right) 
\rightarrow \frac{1}{\sqrt{2 \pi}}\int_{-\infty }^{ \infty}h (t) e ^{-t^2 /2}dt,
\end{align*}
as $x \rightarrow \infty$, where $\pi(x)$ denotes the prime-counting function. In this paper, we will obtain generalisations of this statement over various sets of primes, families of elliptic curves, and for the product of traces of Frobenius for pairs of elliptic curves, as well as an analogue in the modular form setting.

\subsection{Statement of results}

We first establish some notation. Following \cite{ss}, for $f_i(Z),g_i(Z)\in \Z[Z]$, we require
\begin{align*}
	\Delta_i(Z)	&:=-16(4f_i(Z)^3+27g_i(Z)^2)\neq 0,\\
	j_i(Z)		&:=\frac{-1728(4f_i(Z))^3}{\Delta_i(Z)}\in \Q(Z)\sm \Q.
\end{align*}
$E_i (Z)$ will denote the elliptic curve over $\Q(Z)$ represented by the equation 
\begin{align*}
Y^2=X^3+f_i(Z)X+g_i(Z).
\end{align*}
Given some family of sets $\S(T)\subset \Q$ we define
\begin{align*}
R_{\mathcal{S}(T),p}(w) := \# \{u/v\in \mathcal{S}(T):(u,v)=1,p\nmid v,u\equiv vw \pmod p\}.
\end{align*}
In this introduction, unless otherwise stated, we constrain $\mathcal{S}$ to be such that  
\begin{align*}
R_{\S(T),p}(w)=\#\S(T)p^{-1}+O(\#\S(T)p^{-2}+F(T))
\end{align*}
for some suitable $F(T)$, uniformly in $w$.

Lastly, for any integers $\tau,\tau'$ we denote,
\begin{align*}
\pi_{E(Z)}(\tau;x) &= \# \{p\leq x: a_p(E(Z)) =\tau\},\\
\pi_{E(Z),E'(Z')}(\tau,\tau';x)	&= \#\set{p\leq x: a_p(E(Z))  = \tau, a_p(E'(Z')) = \tau'}.
\end{align*}
We now present the main theorems of the paper.

\begin{theorem}\label{prop:set}
Let $\A,\A_1,\A_2$ be prime-indexed sequences and let $\tau,\tau_1,\tau_2$ be integers. 
	(i) If $R_{\S(T),p}(w)\ll \#\S(T)p^{-1}+F(T)$ for some $F(T)$ then for all $0\leq \delta\leq 1/2$ we have
	\begin{align*}
		&\sum_{\substack{t\in \S(T)\\\Delta(t)\neq 0}}\pit{t}(\A;x)\\
			&\qquad\ll	\pw{
									\#\S(T)\frac{x^{1/2}}{\log x}+F(T)\frac{x^{3/2}}{\log x}			&		\A=\tau,\\
									\#\S(T)\frac{x^{1/4}}{\log x}+F(T) \frac{x^{5/4}}{\log x}			&		\A(p)=\pm \floor{2\sqrt{p}},\\
									\#\S(T) x^{1/4+\delta/2}\log_2  x+F(T) x^{5/4+\delta/2}\log_2  x	&		\A(p)=\pm 2\sqrt{p}+O\la(p^{\delta}\ra),\\
									\#\S(T) x^{1/2}\log_2  x+F(T) x^{3/2}\log_2  x				&		\oth.
							}
	\end{align*}
	
	(ii) 	If $R_{\S(T),p}(w)\ll \#\S(T) p^{-1}+F(T)$ for some $F(T)$ then for $0\leq \delta_1,\delta_2\leq 1/2$ we have
			\begin{align*}
				&\sum_{\substack{t_1,t_2\in \S(T)\\\Delta_1(t_1)\neq 0\\ \Delta_2(t_2)\neq 0}}\pitt{t_1}{t_2}(\A_1,\A_2;x)\\
				&\ll \pw{	
							\#\S(T)^2\log_2  x+\#\S(T)F(T)x+F(T)^2x^2
																													&\!\! 		\A_i=\tau_i,\\	
							\#\S(T)^2+\#\S(T)F(T)x^{1/2}\log_2  x+F(T)^2x^{3/2}\log_2  x										
																													&\!\!		\A_i(p)=\pm \floor{2\sqrt p},\\				
							\#\S(T)^2+\Bigl[\#\S(T)F(T)+F^2(T)x\Bigr]x^{\frac{1\!+\!\delta_1\!+\!\delta_2}{2}}\log x\log_2^2  x
																													&\!\!		\A_i(p)=\pm 2\sqrt{p}+O\la(p^{\delta_i}\ra),\\
							\Bigl[\#\S(T)^2\log x\log_2 x+\#\S(T)F(T)x+F(T)^2 x^2\Bigr]\log x\log_2^2  x
																													&\!\!		\oth.
						}
			\end{align*}
\end{theorem}

\begin{theorem} 
\label{prop:setconstlin}\ \\ 
(i) Suppose $\deg j=1$. 
For all $c>0$ we have
\begin{align*}
\sum_{\substack{t\in \S(T)\\\Delta(t)\neq 0}}\pit{t}(\tau;x)=\#\S(T) C_{\tau} \frac{x^{1/2}}{\log x} +O\la( \#\S(T)\frac{x^{1/2}}{\log^c x}+F(T)\frac{x^{3/2}}{\log x}\ra),
\end{align*}
for a constant $C_\tau$.

\label{prop:set2constlin}
(ii) Suppose $\deg j_1=\deg j_2=1$. 
We have
\begin{align*}
	\sum_{\substack{t_1,t_2\in \S(T)\\\Delta_1(t_1)\neq 0 \\ \Delta_2(t_2) \neq 0}}\pitt{t_1}{t_2}(\tau_1,\tau_2;x)
	 =\#\S(T)^2 C_{\tau_1,\tau_2}\log_2  x	+	O\la(	\#\S(T)	+	F(T) x	\ra)^2,
\end{align*}
for a suitable constant $C_{\tau_1,\tau_2}$.
\end{theorem}

One can also consider the case of extremal primes:

\begin{theorem}
\label{prop:setextremallin}
Given the prime-indexed sequence $\A(p)=\pm \floor{2\sqrt p}$ and $\deg j=1$, we have
$$\sum_{\substack{t\in \S(T)\\\Delta(t)\neq 0}}\pit{t}(\A;x)=\#\S(T)\frac{8x^{1/4}}{3\pi\log x}+O\la(\#\S(T)\frac{x^{1/4}}{\log^2 x}+F(T)\frac{x^{5/4}}{\log x}\ra).$$
\end{theorem}

Let a \textit{near-permutation rational for $p$} be a rational function that is injective on its non-singular points in $\F_p$. 
Examples of near-permutation rationals are permutation polynomials which are polynomials that permute $\F_p$.
Let $\pit{t}(\tau;x,\congval,\modval)$ count the number of primes $p \leq x$ such that $a_p(E(t))=\tau$ and $p \equiv \congval \pmod \modval$. Then we have

\begin{theorem}
\label{prop:perm}\ 

			(i) Suppose $j(Z)$ is a near-permutation rational for all $p\equiv \congval\pmod\modval$ such that $p>x_0$ for some $x_0$.
		    For all $c>0$ we have
			$$\sum_{\substack{t\in \S(T)\\\Delta(t)\neq 0}}\pit{t}(\tau;x,\congval,\modval) = \#\S(T)C_{\tau,\congval,\modval}\frac{x^{1/2}}{\log x}+O\la(\#\S(T)\frac{x^{1/2}}{\log^c x}+F(T)\frac{x^{3/2}}{\log x}\ra),$$
            for some constant $C_{\tau,\nu,\omega}$.

\label{prop:perm2}
(ii) Suppose $j_1(Z), j_2(Z)$ are both near-permutation rationals for all $p\equiv \congval \pmod \modval$  such that $p>x_0$ for some $x_0$.
If $2\nmid \tau$ or $\tau=0$ we have
\begin{align*}
\sum_{\substack{t_1,t_2\in \S(T)\\\Delta_1(t_1)\neq 0\neq \Delta_2(t_2)}}\pitt{t_1}{t_2}(\tau,\tau;x,\congval,\modval)
 =\#\S(T)^2 C \log_2  x	+	O\la(\#\S(T)+F(T)x\ra)^2.
\end{align*}
for a suitable constant $C=C_{\tau,\tau,\congval,\modval}$.
\end{theorem}

We also consider `thinner' families with the use of exponential functions.

\begin{theorem}
\label{prop:exp}\ 

(i) Fix $b\in \Z\sm \set{0}$, and let $f,g$ be such that $j(Z)=Zb^Z$, then for all $c>0$ we have
$$\sum_{\substack{t\in \S(T)\\\Delta(t)\neq 0}}\pit{t}(\tau;x) =\#\S(T)C_\tau \frac{x^{1/2}}{\log x} + O\la(\#\S(T)\frac{x^{1/2}}{\log^c x}+F(T)\frac{x^{5/2}}{\log x}\ra).$$

\label{prop:exp2}
(ii) Fix $b_1,b_2\in \Z\sm\set{0}$ and let $f_i,g_i$ be such that $j_i(Z)=Zb_i^Z$.
For all $c>0$ we have
$$\sum_{\substack{t_1,t_2\in \S(T)\\\Delta_1(t_1)\neq 0\neq \Delta_2(t_2)}}\pitt{t_1}{t_2}(\tau_1,\tau_2;x) =\#\S(T)^2C_{\tau_1,\tau_2}\log_2  x	
+ O(\#\S(T)+F(T)x^2)^2.$$
\end{theorem}

Lastly, we present two central limit theorems. Given any infinite set $P$ of primes, let $P_x :=\{p \leq x \mid p \in P\}$, and let $\pi_P (x) := \# P_x$. We say that a polynomial is a \textit{permutation polynomial for $P$} if it is a permutation polynomial for each prime in $P$. We also let $\widetilde{a}_{E}(p) := a_{E}(p) / \sqrt{p}$ be the normalised trace of Frobenius for elliptic curve $E$.
Given two functions $\phi$ and $\psi$ on the integers, define the following family of pairs of elliptic curves, where $A$ and $B$ are positive integers
\begin{align*}
T_{\phi,\psi}(A,B) := \{(E_1, E_2) \mid E_1 \not \simeq E_2, E_i= E(\phi(a_i),\psi(b_i)),  |a_i|\leq A, |b_i| \leq B, i = 1,2\}.
\end{align*}

\begin{theorem}
Let $h$ be any continuous function on $\R$ and let $P$ be an infinite set of primes. Either:\\ 
\label{pclt} (i) let $\phi$ and $\psi$ are permutation polynomials for $P$, or\\
\label{eclt} (ii) for integers $a,b,c,d,h,k$ (where $a,c,d,j$ are non-zero), let $\phi(n) = (an + b)c^n$ and $\psi(n) = (dn + h)k^n$ be functions on $\Z$, where $P$ is such that $p \nmid abcdhk$ for all $p \in P$.\\
Then, provided  $\log A/\log x$ and $\log B/\log x  \rightarrow \infty$ as $x \rightarrow \infty$, we have 
\begin{align*}
\frac{1}{\# T(A,B)}\sum_{(E,E')\in T (A,B)}
h \left(\frac{\sum_{p \in P_x} \widetilde{a}_{E}(p)\widetilde{a}_{E'}(p)}{\sqrt{\pi_P(x)}}\right) 
\rightarrow \frac{1}{\sqrt{2 \pi}}\int_{-\infty }^{ \infty}h (t) e ^{-t^2 /2}dt,
\end{align*}
as $x \rightarrow \infty$.
\end{theorem}

\begin{remark}
An example in the case of part (i) is $\phi = x^5 + 5x^3 + 5x,$ $\psi=x ^5 $ and $P$ being any infinite subset of $\{p \mid p \equiv 2,4 \pmod 5\}$ \cite{MP13}.
\end{remark}

Let $\mathcal{F}_{k,N}$ be the set of normalised cuspidal Hecke eigenforms of weight $k$ and level $N$, and let $\lambda_f (n)$ denote the $n$th Hecke eigenvalue of a modular form $f$.

\begin{theorem}\label{mf-clt}
For $kN$ such that $\log kN / \log x \rightarrow \infty$ as $x \rightarrow \infty$,
\begin{align*}
\frac{1}{(\# \mathcal{F}_{k,N})^2}
\sum_{\text{ distinct }f_1, f_2 \in \mathcal{F}_{k,N}} 
h \left(\frac{\sum_{p \in P_x} \lambda_{f_1}(p) \lambda_{f_2}(p)}{\sqrt{\pi_P(x)}}\right) \rightarrow \frac{1}{\sqrt{2 \pi}}\int_{-\infty }^{ \infty}h (t) e ^{-t^2 /2}dt,
\end{align*}
as $x \rightarrow \infty$.
\end{theorem}

We also refer the reader to Section \ref{sec:ex} for further concrete examples of asymptotic results related to the above theorems.\\

This paper is organised as follows. In Section \ref{sec:prelim}, we set up notation and introduce some lemmas. In Section \ref{sec:basic}, we establish lemmas on the asymptotic behaviour on series involving Hurwitz numbers and obtain some upper bounds on Lang--Trotter type questions. In Section \ref{sec:asym}, we establish some asymptotic results for certain types of families of elliptic curves, including Central Limit Theorems for pairs of elliptic curves for two different families, and a similar result for modular forms. In Section \ref{sec:res}, we obtain bounds on Hurwitz number related series under congruence conditions. In Section \ref{sec:ex}, we obtain further concrete examples of asymptotic results on average. 

\subsection*{Acknowledgements}
The authors would like to thank Amir Akbary for his helpful comments. The first author was supported in part by an NSERC USRA fellowship. The second author was supported in part by an NSERC Discovery Grant.

\section{Preliminaries and Notation}\label{sec:prelim}
        We begin by collecting relevant notation in this section (including that introduced in the previous section).

		We set up families of curves similar to \cite{ss}.		
		Let $f_i(Z),g_i(Z)\in \Z[Z]$ be such that
		\begin{align*}
			\Delta_i(Z)	&:=-16(4f_i(Z)^3+27g_i(Z)^2)\neq 0,\\
			j_i(Z)		&:=\frac{-1728(4f_i(Z))^3}{\Delta_i(Z)}\in \Q(Z)\sm \Q,
		\end{align*}
		and consider the elliptic curves
		$$E_i(Z):\qquad Y^2=X^3+f_i(Z)X+g_i(Z)$$
		over $\Q(Z)$.
		The degree of the rational function $j_i(Z)$ is $\deg j_i=\max(\deg q_i,\deg r_i)$ where $j_i(Z)=\smash{\frac{q_i(Z)}{r_i(Z)}}$ for $q_i(Z),r_i(Z)\in \Z[Z]$ such that $(q_i(Z),r_i(Z))=(1)$.
		In \cref{sec:exp} we briefly relax the condition $f_i(Z),g_i(Z)\in \Z[Z]$ and consider more general families of curves where the coefficient functions $f_i(Z), g_i(Z)$ are not polynomials.

		Let $\A_i:\Spec(\Z)\sm \set{0}\to \Z$ and let 
		\begin{align*}
			\pi_{E_i(z_i)}(\A_i;x)				&=\#\set{p\leq x:a_{i,z_i,p}=\A_i(p)}\\
			\pi_{E_1(z_1),E_2(z_2)}(\A_1,\A_2;x)		&=\#\set{p\leq x:a_{1,z_1,p}=\A_1(p),a_{2,z_2,p}=\A_2(p)},
		\end{align*}
		where $a_{i,z_1,p}=a_p(E_i(z_i))$.
		Let $N_i(z_i)$ be the conductor of $E_i(z_i)$ and let
		\begin{align*}
			\S(T)					&\subset \Q,\\
			\P					&=\max\set{p:\exists u/v\in \Q,p\mid v,p\nmid N_1(u/v)N_2(u/v)},\\
			R_{\S(T),p}(w)			&=\#\set{u/v\in \S(T):(u,v)=1,p\nmid v,u\equiv vw\pmod p},\\
			\norm{R_{\S(T),p}}_\infty		&=\max_{0\leq w\leq p-1}R_{\S(T),p}(w).
		\end{align*}
		In \cref{sec:ex} we consider the cases when $\S(T)$ is the integers $\I(T)=\set{1,\dots,T}$ or the Farey fractions $\cF(T)=\set{\frac{u}{v}\in \Q:(u,v)=1,1\leq u,v\leq T}$ and compare our results to \cite{ss}.
		By \cite{ss} we have $\P<\infty$.
		We will consider families of sets $\S(T)$ which are relatively evenly distributed modulo primes, that is, either
		$$R_{\S(T),p}(w)=\#\S(T)p^{-1}+O(\#\S(T)p^{-2}+F(T))$$
		uniformly in $w$ for some $F(T)$ or
		$$\norm{R_{\S(T),p}}_\infty	\ll	\#\S(T)p^{-1}+F(T)$$
		depending on whether we want an asymptotic formula or just an upper bound.
		Finally, let $\log_2 x=\log \log x$ and $\pi_{1/2}(x)=\int_2^x\frac{\differ t}{2\sqrt t \log t}\sim\frac{x^{1/2}}{\log x}$.
		\begin{lem}\label{lem:iso}
			There exists $x_0$ such that for all $p>x_0$ we have 
			$$\sum_{\substack{w\in \F_p\\\Delta(w)\not\equiv 0\pmod p\\a_{w,p}=\A(p)}}1\leq (\deg j)\H{\A(p)}{p}.$$
			\begin{proof}
				Let $j(Z)=\frac{q(Z)}{r(Z)}$ for coprime $q(Z),r(Z)\in \Z[Z]$ and define $\deg j=\max(\deg q,\deg r)$.
				Since $j(Z)\in \Q(Z)\sm \Q$, there exists $x_0$ such that $p>x_0$ implies $j(Z)\mod p\in \F_p(Z)\sm \F_p$, so the number of solutions to $j(Z)\equiv a\pmod p$ is at most $\deg j$.
				Let $\jp{\A(p)}$ be the set of $j$-invariants of the isomorphism classes of elliptic curves over $\mathbb{F}_p$ with Frobenius trace $\A(p)$.
				\begin{align*}
					\sum_{\substack{w\in \F_p\\\Delta(w)\not\equiv 0\pmod p\\a_{w,p}=\A(p)}}1	=\sum_{\substack{w\in \F_p\\\Delta(w)\not\equiv 0\pmod p\\j(w)\in \jp{\A(p)}}}1
																		&=\sum_{a\in \jp{\A(p)}}\sum_{\substack{w\in \F_p\\\Delta(w)\not\equiv 0\pmod p\\j(w)\equiv a\pmod p}}1\\
																		&\leq \sum_{a\in \jp{\A(p)}}(\deg j)\\
																		&=(\deg j)\H{\A(p)}{p}.\qedhere
				\end{align*}
			\end{proof}
		\end{lem}
		\begin{lem}\label{lem:M}
			For some $x_0>0$, we have $p>x_0$ implies
			$$\sum_{\substack{t\in \S(T)\\\Delta(t)\not\equiv 0\pmod p\\a_{t,p}=\A(p)}}1\leq \norm{R_{\S(T),p}}_\infty (\deg j) \H{\A(p)}{p}.$$
			\begin{proof}
				Let $x_0>\P$ be as in \cref{lem:iso}.
				By \cite{ss} we have
				\begin{align*}
							\sum_{\substack{t\in \S(T)\\\Delta(t)\not\equiv 0\pmod p\\a_{t,p}=\A(p)}}1
							&=\sum_{\substack{w\in \F_p\\\Delta(w)\not\equiv 0\pmod p\\a_{w,p}=\A(p)}}R_{\S(T),p}(w)\\
							&\leq \norm{R_{\S(T),p}}_\infty\sum_{\substack{w\in \F_p\\\Delta(w)\not\equiv 0\pmod p\\a_{w,p}=\A(p)}}1\\
							&\leq \norm{R_{\S(T),p}}_\infty (\deg j)
							=\H{\A(p)}{p}.\qedhere
				\end{align*}
			\end{proof}
		\end{lem}
		
\section{Basic Bounds\label{sec:basic}}
		\begin{lem}\label{lem:hur}
			For all $\tau\in \Z$, $c>0$, and $0\leq \delta \leq 1/2$ there exists $C_\tau>0$ such that
			\begin{align*}
				\sum_{p\leq x}\frac{\H{\A(p)}{p}}{p}
						&=	\pw{
									C_\tau \pi_{1/2}(x)+O\la(\frac{x^{1/2}}{\log^c x}\ra)					&	\A=\tau,\\
									\frac{8}{3\pi} \frac{x^{1/4}}{\log x}+O\la(\frac{x^{1/4}}{\log^2 x}\ra)	&	\A(p)=\pm \floor{2\sqrt p},\\
									O(x^{1/4+\delta/2}\log_2  x)								&	\A(p)=\pm 2\sqrt{p}+O\la(p^{\delta}\ra),\\
									O(x^{1/2}\log_2  x)									&	\oth.
							}
			\end{align*}
			\begin{proof}
				The case $\A=\tau$ is proved in \cite{dp} with an explicit formula for the constant $C_\tau$.
				The case $\A(p)=\pm \floor{2\sqrt p}$ is proved in \cite{gj}.
				For the other two cases, for $\abs{\A(p)}< 2\sqrt p$ by \cite{g} we have
				\begin{align*}
					\H{\A(p)}{p}		&=2\sum_{\substack{df^2=\A(p)^2-4p\\d\equiv 0,1\pmod 4}}\frac{h(d)}{w(d)}\\
								&=\sum_{\substack{df^2=\A(p)^2-4p\\d\equiv 0,1\pmod 4}}\frac{\sqrt{\abs{d}}}{\pi}L(1,\chi_{d})\\
								&\ll \log p\sum_{\substack{df^2=\A(p)^2-4p\\d\equiv 0,1\pmod 4}}\frac{\sqrt{\abs{\A(p)^2-4p}}}{f}\\
								&\ll (\log p)\sqrt{\abs{\A(p)^2-4p}}\sum_{f\mid \abs{\A(p)^2-4p}}\frac{1}{f}\\
								&\ll (\log p)\sqrt{\abs{\A(p)^2-4p}} \log_2 \abs{\A(p)^2-4p}.
				\end{align*}
				For $\abs{\A(p)}\geq 2\sqrt p$ we have $\H{\A(p)}{p}=0$.
				Then in general we have
				$$\H{\A(p)}{p}\ll p^{1/2}\log p \log_2  p.$$
				If $\A(p)=2\sqrt{p}+O\la(p^{\delta}\ra)$ for some $0\leq \delta \leq 1/2$ we have.
				\begin{align*}
					\H{\A(p)}{p}						&\ll p^{1/4+\delta/2}\log p \log_2  p,\\
					\sum_{p\leq x}\frac{\H{\A(p)}{p}}{p}	
												&\ll \sum_{p\leq x}p^{-3/4+\delta/2}\log p\log_2  p\\
												&\ll x^{1/4+\delta/2}\log_2  x.\qedhere
				\end{align*}
			\end{proof}
		\end{lem}
			\begin{proof}[Proof of \cref{prop:set}(i)] 
				Let $x_0$ be as in \cref{lem:M}.
				We have
				\begin{align*}
					\sum_{\substack{t\in \S(T)\\\Delta(t)\neq 0}}\pit{t}(\A;x)
							&\leq x_0\#\S(T)+\sum_{x_0<p\leq x}\sum_{\substack{t\in \S(T)\\\Delta(t)\not\equiv 0\pmod p\\a_{t,p}=\A(p)}}1\\
							&\leq x_0\#\S(T)+\sum_{x_0<p\leq x}\norm{R_{\S(T),p}}_\infty (\deg j) \H{\A(p)}{p}\\
							&\ll x_0\#\S(T)+\sum_{x_0<p\leq x}\la[\#\S(T)p^{-1}\H{\A(p)}{p}+F(T)\H{\A(p)}{p}\ra].
				\end{align*}
				The result then follows from \cref{lem:hur} (noting again that $\pi_{1/2}(x)\sim \frac{x^{1/2}}{\log x}$).
			\end{proof}
		\begin{lem}
		    We have
		    \begin{align*}
		        \sum_{\substack{t\in \S(T)\\\Delta(t)\neq 0\\E(t)\text{\;has CM}}}\pit{t}(\A;x)\ll x.
		    \end{align*}
		    \begin{proof}
		        Let $J_{\mathrm{CM}}$ be the set of 13 $j$-invariants of CM elliptic curves over $\Q$.
		        We have
		        \begin{equation*}
		            \sum_{\substack{t\in \S(T)\\\Delta(t)\neq 0\\E(t)\text{\;has CM}}}\pit{t}(\A;x)\leq x\sum_{\substack{t\in \C\\j(t)\in J_{\mathrm{CM}}}}1\leq 13(\deg j)x.\qedhere
		        \end{equation*}
		    \end{proof}
		\end{lem}
		The above bound is smaller that the $F(T)$ terms \cref{prop:set}(i) as long as $F(T)\gg 1$.

		\begin{lem}\label{lem:hur2}
			For $0\leq \delta_1,\delta_2<1/2$ we have
			\begin{align*}
				\sum_{p\leq x}\frac{\H{\A_1(p)}{p}\H{\A_2(p)}{p}}{p^2}
						&=	\pw{
									C_{\tau_1,\tau_2} \log_2  x+O(1)							&	\A_i=\tau_i,\\
									O(1)												&	\A_i(p)=\pm \floor{2\sqrt p},\\
									O(1)												&	\A_i(p)=\pm 2\sqrt{p}+O\la(p^{\delta_i}\ra),\\
									O(\log^2 x\log_2^3 x)								&	\oth,
							}
			\end{align*}
			\begin{align*}
				\sum_{p\leq x}\frac{\H{\A_1(p)}{p}\H{\A_2(p)}{p}}{p}
						&\ll	\pw{
									x												&	\A_i=\tau_i,\\
									x^{1/2}\log_2  x										&	\A_i(p)=\pm \floor{2\sqrt p},\\
									x^{\frac{1+\delta_1+\delta_2}{2}}\log x\log_2^2 x					&	\A_i(p)=\pm 2\sqrt{p}+O\la(p^{\delta_i}\ra),\\
									x \log x\log_2^2 x										&	\oth.
							}
			\end{align*}
			\begin{proof}
				The first case
				$$\sum_{p\leq x}\frac{\H{\tau_1}{p}\H{\tau_2}{p}}{p^2}=C_{\tau_1,\tau_2} \log_2  x+O(1)$$
				is proved in \cite{ap} and by Abel's summation formula we have
				$$\sum_{p\leq x}\frac{\H{\tau_1}{p}\H{\tau_2}{p}}{p}=O\la(x\ra).$$

				If $\A_i(p)=\pm 2\sqrt{p}+O\la(p^{\delta}\ra)$ then by the bound $\H{\A_i(p)}{p}\ll p^{1/4+\delta_i/2}\log p\log_2  p$.
				\begin{align*}
					\sum_{p\leq x}\frac{\H{\A_1(p)}{p}\H{\A_2(p)}{p}}{p^2}	&\ll 1.
				\end{align*}
				If $\A_i(p)=\pm \floor{2\sqrt p}$ then by \cref{lem:hur} we have
				\begin{align*}
					\sum_{p\leq x}\frac{\H{\A_1(p)}{p}\H{\A_2(p)}{p}}{p}	&\ll \sum_{p\leq x}\frac{\H{\A_i(p)}{p}}{p}p^{1/4}\log p\log_2  p\\
														&\ll x^{1/2}\log_2  x.
				\end{align*}
				If $\A_i(p)=\pm 2\sqrt{p}+O\la(p^{\delta_i}\ra)$ then by Abel's summation formula we have
				\begin{align*}
					\sum_{p\leq x}\frac{\H{\A_1(p)}{p}\H{\A_2(p)}{p}}{p}	&\ll \sum_{p\leq x}p^{\frac{-1+\delta_1+\delta_2}{2}}\log^2 p\log_2^2  p\\
														&\ll x^{\frac{1+\delta_1+\delta_2}{2}}\log x\log_2^2 x.
				\end{align*}
				
				Finally, in the general case, we have
				\begin{align*}
					\sum_{p\leq x}\frac{\H{\A_1(p)}{p}\H{\A_2(p)}{p}}{p^2}	&\ll \sum_{p\leq x}p^{-1}\log^2 p\log_2^2 p\\
														&\ll \log^2 x\log_2^3 x,\\
					\sum_{p\leq x}\frac{\H{\A_1(p)}{p}\H{\A_2(p)}{p}}{p}	&\ll \sum_{p\leq x}\log p\log_2  p\\
														&\ll x \log x\log_2^2 x.\qedhere
				\end{align*}
			\end{proof}
		\end{lem}
			\begin{proof}[Proof of \cref{prop:set}(ii)]
				Let $x_0$ be as in \cref{lem:M}.
				We have
				\begin{align*}
					&\sum_{\substack{t_1,t_2\in \S(T)\\\Delta_1(t_1)\neq 0\neq \Delta_2(t_2)}}\pitt{t_1}{t_2}(\A_1,\A_2;x)\\
							&\;\leq x_0\#\S(T)^2+\sum_{x_0<p\leq x}\sum_{\substack{t_1,t_2\in \S(T)\\\Delta(t_1)\not\equiv 0\not\equiv \Delta(t_2)\pmod p\\a_{1,t,p}=\A_1(p)\\a_{2,t,p}=\A_2(p)}}1\\
							&\;\leq x_0\#\S(T)^2+\sum_{x_0<p\leq x}\norm{R_{\S(T),p}}_\infty^2 (\deg j_1)(\deg j_2) \H{\A_1(p)}{p}\H{\A_2(p)}{p}\\
							&\;\ll \#\S(T)^2+\sum_{x_0<p\leq x}\la(\#\S(T)^2p^{-1}+F(T)p^{-1}\ra)\H{\A_1(p)}{p}\H{\A_2(p)}{p}.
				\end{align*}
				The result then follows from \cref{lem:hur2}.
			\end{proof}
		\begin{lem}
		    We have
		    $$\sum_{\substack{t_1,t_2\in \S(T)\\\Delta_1(t_1)\neq 0\neq \Delta_2(t_2)\\E_1(t_1)\text{\;or\;}E_2(t_2)\text{\;has CM}}}\pitt{t_1}{t_2}(\A_1,\A_2;x)\ll \#\S(T)x^2 \log_2 x+F(T)x^{3/2}\log_2 x.$$
		    \begin{proof}
		        Again let $J_{\mathrm{CM}}$ be the set of 13 $j$-invariants of CM elliptic curves over $\Q$.
		        By \cref{prop:set}(i) we have
		        \begin{align*}
		           & \sum_{\substack{t_1,t_2\in \S(T)\\\Delta_1(t_1)\neq 0\neq \Delta_2(t_2)\\E_1(t_1)\text{\;or\;}E_2(t_2)\text{\;has CM}}}\pitt{t_1}{t_2}(\A_1,\A_2;x)\\
		            &\qquad\leq \la(\sum_{\substack{t_1\in \C\\j(t_1)\in J_{\mathrm{CM}}}}1\ra)\la( \sum_{\substack{t_2\in \S(T)\\\Delta_2(t_2)\neq 0}}\pi_{E_2(t_2)}(\A_2;x)\ra)
		               +\la(\sum_{\substack{t_2\in \C\\j(t_2)\in J_{\mathrm{CM}}}}1\ra)\la( \sum_{\substack{t_1\in \S(T)\\\Delta_1(t_1)\neq 0}}\pi_{E_1(t_2)}(\A_1;x)\ra)\\
		            &\qquad\ll \#\S(T)x^{1/2}\log_2 x+F(T)x^{3/2}\log_2 x.\qedhere
		        \end{align*}
		    \end{proof}
		\end{lem}
		Again, the above is smaller than the $\#\S(T)F(T)$ and $F(T)^2$ terms in \cref{prop:set}(ii) if $F(T)\gg 1$.
		\begin{lem}\label{lem:isogenous}
		    \begin{align*}
		       &\sum_{\substack{t_1,t_2\in \S(T)\\\Delta_1(t_1)\neq 0\neq \Delta_2(t_2)\\E(t_1)\sim  E(t_2)\text{\;over\;}\bar{\Q}}}\pitt{t_1}{t_2}(\A_1,\A_2;x)\\
		        &\qquad\ll W(T)^8 \log W(T)\la[\#\S(T) x^{1/2}\log_2 x+F(T) x^{3/2}\log_2 x\ra].
		    \end{align*}
		    where 
		    $$W(T)=1+\log\max_{\substack{\frac{u}{v}\in \S(T)\\(u,v)=1}}\max(\abs{u},\abs{v}).$$
		    \begin{proof}
		       This is a modified version of the proof of \cite[Lemma 1]{fm2}.
		       For a curve $E:y^2=x^3+\alpha x+\beta$ with $\alpha,\beta\in \Q$ the logarithmic Weil height of $E$ is
		       $$w(E)=\max(1,h(4\alpha),h(4\beta))$$
		       where $h(a/b)$ for $(a,b)=1$ is given by
		       $$h(a/b)=\log\max(\abs{a},\abs{b}).$$
		       By \cite{mw}, if $E_1(t_1)$ and $E_2(t_2)$ are isogenous over $\bar \Q$ there exists an cyclic isogeny of degree $n\leq cw(E_i(t_i))^4$ for some absolute constant $c$.
		       Such an isogeny exists if and only if $\Phi_n(j_1(t_1),j_2(t_2))=0$ where $\Phi_n(X,j)$ has degree $\psi(n)=n\prod_{p\mid n}\la(1+\frac 1p\ra)=O(n\log 2n)$.
		       For any $t_2\in \S(T)$ we have
		       \begin{align*}
		           \#\set{t_1\in \S(T):E_1(t_1)\sim E_2(t_2)\text{\;over\;}\bar \Q}
		           &\leq \sum_{d\leq cw(E_2(t_2))^4}\sum_{\substack{t_1\in \C\\\Phi_n(j_1(t_1),j_2(t_2))=0}}1\\
		           &\ll w(E_2(t_2))^8 \log(2w(E_2(t_2))^4).
		       \end{align*}
		       A similar bound holds if we fix $E_1(t_1)$ and count the number of isogenous $E_2(t_2)$ for $t_2\in \S(T)$.
		       We have
		       \begin{align*}
		           w(E_2(u/v))
		           \ll \max\set{1,\log \max(\abs{u},\abs{v})^{\max(\deg f_2,\deg g_2)}}
		           \ll W(T).
		       \end{align*}
		       The same bound holds for $w(E_1(t_1))$ (though the constant depends on the polynomials $f_i,g_i$).
		       By \cref{prop:set}(i) we have
		       \begin{align*}
		           &\sum_{\substack{t_1,t_2\in \S(T)\\\Delta_1(t_1)\neq 0\neq \Delta_2(t_2)\\E(t_1)\sim E(t_2)\text{\;over\;}\bar{\Q}}}\pitt{t_1}{t_2}(\A_1,\A_2;x)\\
		           &\qquad\ll W(T)^8 \log W(T)\sum_{\substack{t_2\in \S(T)\\\Delta_2(t_2)\neq 0}}\pi_{E_2(t_2)}(\A_2;x)\\
		           &\qquad\ll W(T)^8 \log W(T)\la[\#\S(T) x^{1/2}\log_2 x+F(T) x^{3/2}\log_2 x\ra].\qedhere
		       \end{align*}
		    \end{proof}
		\end{lem}
		For the sets in \cref{sec:ex} we have $W(T)\ll \log T$, however this sum over isogenous pairs becomes significant compared to the bounds \cref{prop:set}(ii) for families such as $\S(T)=\set{1,2,4,8,\dots,2^T}$.

\section{Asymptotics for Special Families\label{sec:asym}}

	\subsection{Near-Permutation Rational Families\label{sec:perm}}
		\begin{defn}
			Let $q(Z)\in \F_p(Z)$ and let $S$ be its set of singularities in $\F_p$. Then
			$q$ is called a near-permutation rational function for $p$ if $q:\F_p\sm S\to \F_p$ is injective.
		\end{defn}
		Examples of near-permutation rational functions are $Z^3\in \F_p(Z)$ if $p\equiv 2\pmod 3$ or $\frac{aZ+b}{cZ+d}\in \F_p(Z)$ if $ad-bc\not\equiv 0\pmod p$.
		\begin{lem}\label{lem:isoperm}
			If $j(Z)\mod p$ is a near-permutation rational function for all $p\equiv \congval \pmod \modval$ for some $(\congval,\modval)=1$ then $\sum_{\substack{w\in \F_p\\\Delta(w)\not\equiv 0\pmod p\\a_{w,p}=\tau}}1= \H{\tau}{p}+O(1)$.
			\begin{proof}
				Let $\jp{\tau}$ be the set of $j$-invariants of the isomorphism classes of elliptic curves over $\mathbb{F}_p$ with Frobenius trace $\tau$.
				We have
				\begin{align*}
					\sum_{\substack{w\in \F_p\\\Delta(w)\not\equiv 0\pmod p\\a_{w,p}=\tau}}1	
						&=O(1)+\sum_{\substack{w\in \F_p\\j(w)\mod p\in \jp{\tau}}}1\\
						&=O(1)+\sum_{a\in \jp{\tau}}\sum_{\substack{w\in \F_p\\j(w)\equiv a\pmod p}}1\\
						&=O(1)+\sum_{a\in \jp{\tau}}1\\
						&=O(1)+ \H{\tau}{p}.\qedhere
				\end{align*}
			\end{proof}
		\end{lem}
		Note that 
		\begin{align*}
			\sum_{p\leq x}\frac{\H{\A(p)}{p}}{p^2}				&\ll \sum_{p\leq x}\frac{p^{1/2}\log^2 p}{p^2}\ll 1,\\
			\sum_{p\leq x}\frac{\H{\A_1(p)}{p}\H{\A_2(p)}{p}}{p^3}		&\ll \sum_{p\leq x}\frac{p\log^4 p}{p^3}\ll 1.
		\end{align*}

		\begin{proof}[Proof of Theorem \ref{prop:setconstlin}(i)]
				Since $\deg j=1$, $j(Z)\mod p$ is near-permutation for all $p>x_0$ for some $x_0$.
				Following \cref{prop:set}(i) with \cref{lem:hur} and \cref{lem:isoperm} we have 
				\begin{align*}
					\sum_{\substack{t\in \S(T)\\\Delta(t)\neq 0}}\pit{t}(\tau;x)
							&= O(\#\S(T))+\sum_{x_0<p\leq x}\la(\#\S(T) p^{-1}+O(\#\S(T)p^{-2}+F(T))\ra) (\H{\tau}{p}+O(1))\\
							&=\#\S(T) C_{\tau} \pi_{1/2}(x)+O\la( \#\S(T)\frac{x^{1/2}}{\log^c x}\ra)+O\la(F(T)\frac{x^{3/2}}{\log x}\ra).\qedhere
				\end{align*}
			\end{proof}
		
\begin{proof}[Proof of Theorem \ref{prop:set2constlin}(ii)]
Since $\deg j_1=\deg j_2=1$, $j_1(Z),j_2(Z)\mod p$ are near-permutation for all $p>x_0$ for some $x_0$.
Following \cref{prop:set}(ii) with Lemmas \ref{lem:hur}, \ref{lem:hur2}, and \ref{lem:isoperm} we have
\begin{align*}
&\sum_{\substack{t_1,t_2\in \S(T)\\\Delta_1(t_1)\neq 0\neq \Delta_2(t_2)}}\pitt{t_1}{t_2}(\tau_1,\tau_2;x)\\
&\qquad	= O(\#\S(T)^2)+\sum_{x_0<p\leq x}\la(\#\S(T) p^{-1}+O(\#\S(T)p^{-2}+F(T))\ra)^2\\
&\qquad\qquad\qquad\qquad\qquad\qquad\qquad\qquad \times (\H{\tau_1}{p}+O(1))(\H{\tau_2}{p}+O(1))\\
&\qquad= \#\S(T)^2 C_{\tau_1,\tau_2}\log_2  x	+	O\la(	\#\S(T)^2	+	F(T)\#\S(T) x	+	F(T)^2 x^2		\ra).\qedhere
\end{align*}
\end{proof}

\begin{proof}[Proof of Theorem \ref{prop:setextremallin}]
Since $\deg j=1$, $j(Z)\mod p$ is near-permutation for all $p>x_0$ for some $x_0$.
Following \cref{prop:set}(i) with \cref{lem:isoperm}, we have
\begin{align*}
&\sum_{\substack{t\in \S(T)\\\Delta(t)\neq 0}}\pit{t}(\A;x)\\
							&\qquad= O(\#\S(T))+\sum_{x_0<p\leq x}\la(\#\S(T) p^{-1}+O(\#\S(T)p^{-2}+F(T))\ra) (\H{\floor{2\sqrt p}}{p}+O(1))\\
							&\qquad= \#\S(T)\frac{8x^{1/4}}{3\pi\log x}+O\la(\#\S(T)\frac{x^{1/4}}{\log^2 x}+F(T)\frac{x^{5/4}}{\log x}\ra).\qedhere
				\end{align*}
			\end{proof}
		
			If $F(T)=o(x^{-1}\#\S(T))$ then the formulas in Theorems \ref{prop:setconstlin}(i) and \ref{prop:setextremallin} are asymptotic.
			If $F(T)=o(x^{-1}\sqrt{\log_2  x}\#\S(T))$ then the formula in \cref{prop:set2constlin}(ii) is asymptotic.

		Let $(\congval,\modval)=1$ and let $\pi_{E(t)}(\A;x,\congval,\modval)$ be the number of $p\equiv \congval \pmod \modval$ such that $p\leq x$ and $\A(p)=a_{t,p}$.
		Define $\pi_{E(t_1),E(t_2)}(\A_1,\A_2;x,\congval,\modval)$ similarly.
		\begin{prop}\label{prop:avgHcong}
			There exists $C_{\tau,\congval,\modval}>0$ such that for all $c>0$,
			$$\sum_{\substack{p\leq x\\p\equiv \congval\pmod\modval}}\frac{\H{\tau}{p}}{p} = C_{\tau,\congval,\modval}\pi_{1/2}(x)+O\la(\frac{x^{1/2}}{\log^c x}\ra).$$
			\begin{proof}
			    Follows immediately from \cite[Proposition 2.1]{jcong}
			\end{proof}
		\end{prop}
		\begin{prop}\label{prop:avgHHcong}
			If $2\nmid \tau$, there exists $C_{\tau,\tau,\congval,\modval}>0$ such that
			$$\sum_{\substack{x_0<p\leq x\\p\equiv \congval\pmod\modval}}\frac{\H{\tau}{p}^2}{p^2} = C_{\tau,\tau,\congval,\modval}\log_2  x+O(1).$$
		\end{prop}
		The proof of \cref{prop:avgHHcong} is given in Section \ref{sec:res}. 

			\begin{proof}[Proof of Theorem \ref{prop:perm}](i)
				The proof is identical to that of \cref{prop:setconstlin}(i) except for the use of \cref{prop:avgHcong} instead of \cref{lem:hur}
			\end{proof}
		
			\begin{proof}[Proof of Theorem \ref{prop:perm2}](ii)
				The proof is identical to \cref{prop:set2constlin}(ii) except for the use of \cref{prop:avgHHcong} instead of \cref{lem:hur2}.
			\end{proof}

			Note that in all of the results in this section we need $F(T)=o(x^{-1}\#\S(T))$ for the single curve formulas to be asymptotic and $F(T)=o(x^{-1}\sqrt{\log_2  x}\#\S(T))$ for the two-curve formulas to be asymptotic.
			Furthermore, combining Theorems \ref{prop:perm}(i) and \ref{prop:perm2}(ii) with Propositions \ref{prop:set}(i) and \ref{prop:set}(ii), if $2\nmid \tau$ and $F(T)=o(x^{-1}\#\S(T))$ we have
			\begin{align*}
				\sum_{\substack{t\in \S(T)\\\Delta_1(t)\neq 0}}\pit{t}(\tau;x)								&=\Theta\la(\#\S(T)\frac{x^{1/2}}{\log x}\ra),\\
				\sum_{\substack{t_1,t_2\in \S(T)\\\Delta_1(t_1)\neq 0\neq \Delta_2(t_2)}}\pitt{t_1}{t_2}(\tau,\tau;x)		&=\Theta\la(\#\S(T)^2 \log_2  x\ra).
			\end{align*}

	\subsection{Exponential coefficient family}\label{sec:exp}
		For any $h(Z)\in \Z[Z]$, odd positive $m,n$, and any function $j^*(Z)$ let
		\begin{align*}
			f(Z)			&=-3j^*(Z)^{n}(j^*(Z)-1728)^{m}h(Z)^2,\\
			g(Z)			&=2j^*(Z)^{\frac{3n-1}{2}}(j^*(Z)-1728)^{\frac{3m+1}{2}}h(Z)^3.
		\end{align*}
		We have
		\begin{align*}
			j(Z)		&=j^*(Z)\\
			\Delta(Z)	&=1728^2 h(Z)^6j^*(Z)^{3n-1}(j^*(Z)-1728)^{3m}.
		\end{align*}
		\begin{lem}\label{lem:isolam}
			Suppose $b\in \Z\sm \set{0}$ and $j^*(Z)=Zb^Z$ with $f,g$ in the above form.
			For $p>6b$ we have
			$$\sum_{\substack{0\leq w< p(p-1)\\\Delta(w)\not\equiv 0\pmod p\\a_{w,p}=\tau}}1=(p-1)\H{\tau}{p}+O(p).$$
			\begin{proof}
				Note that for $p\nmid 6b$, the roots of $\Delta(Z)\mod p$ are the union of roots of $h(Z),j^*(Z),(j^*(Z)-1728)\mod p$ as subsets of $\Z/p(p-1)\Z$.
				Then there are $O(p)$ roots of $\Delta(Z)\mod p$.
				Note that $b^{w_1}\equiv b^{w_2}\pmod p$ if $w_1\equiv w_2\pmod {p-1}$.
				By the Chinese Remainder Theorem we have
				\spanner{\begin{align*}
					\sum_{\substack{0\leq w< p(p-1)\\\Delta(w)\not\equiv 0\pmod p\\a_{w,p}=\tau}}1	
							&=O(p)+\sum_{a\in \jp{\tau}}\sum_{\substack{0\leq w< p(p-1)\\w b^w\equiv a\pmod p}}1\\
							&=O(p)+\sum_{a\in \jp{\tau}}\sum_{0\leq u< p-1}\sum_{\substack{0\leq v< p\\vb^{u}\equiv a\pmod p}}1\\
							&=O(p)+\sum_{a\in \jp{\tau}}\sum_{0\leq u< p-1}1\\
							&=O(p)+(p-1)\H{\tau}{p}.\qedhere
				\end{align*}}
			\end{proof}
		\end{lem}
		For the following proposition we'll modify some of the earlier notation.
		Let
		\begin{align*}
			\S(T)			&\subset \Z,\\
			R_{\S(T),p}(w)	&=\#\set{t\in \S(T):t\equiv w\pmod {p(p-1)}}
		\end{align*}
		where $R_{\S(T),p}(w)=\#\S(T)[p(p-1)]^{-1}+O(F(T))$ uniformly in $w$ and $p$ for some $F(T)$.
		\begin{proof}[Proof of Theorem \ref{prop:exp}(i)]
				Let $x_0=6b$.
				Following \cref{prop:setconstlin}(i) we have
				\begin{align*}
					&\sum_{\substack{t\in \S(T)\\\Delta(t)\neq 0}}\pit{t}(\tau;x)\\
					&\qquad=O(\#\S(T))+\sum_{x_0<p\leq x}\la(\#\S(T) [p(p-1)]^{-1}+O(F(T))\ra)(p-1)(\H{\tau}{p}+O(1))\\
					&\qquad=\#\S(T)C_\tau \frac{x^{1/2}}{\log x}+O\la(\#\S(T)\frac{x^{1/2}}{\log^c x}+F(T)\frac{x^{5/2}}{\log x}\ra).\qedhere
				\end{align*}
			\end{proof}
		
			\begin{proof}[Proof of Theorem \ref{prop:exp2}(ii)]
				Let $x_0=6b_1b_2$.
				Following \cref{prop:setconstlin}(i) we have
				\begin{align*}
					&\sum_{\substack{t_1,t_2\in \S(T)\\\Delta_1(t_1)\neq 0\neq \Delta_2(t_2)}}\pitt{t_1}{t_2}(\tau_1,\tau_2;x)\\
					&=O(\#\S(T))+\sum_{x_0<p\leq x}\la(\frac{\#\S(T)}{p}+O(pF(T))\ra)^2(\H{\tau_1}{p}+O(1))(\H{\tau_2}{p}+O(1))\\
					&=		 	\#\S(T)^2C_{\tau_1,\tau_2}\log_2  x					+		O(\#\S(T)^2		+		\#\S(T)F(T)x^2	+		F(T)^2 x^4).\qedhere
				\end{align*}
			\end{proof}

			It's interesting that since we are summing over the modulus $p(p-1)$ we need the bound $F(T)=o(x^{-2}\#\S(T))$ to make \cref{prop:exp}(i) asymptotic and $F(T)=o(x^{-2}\sqrt{\log_2  x}\#\S(T))$ to make \cref{prop:exp2}(ii) asymptotic.
			This is much stronger than the bound $F(T)=o(x^{-1}\#\S(T))$ or $F(T)=o(x^{-1}\sqrt{\log_2  x}\#\S(T))$ needed for Theorems \ref{prop:setconstlin}(i), \ref{prop:set2constlin}(ii), \ref{prop:perm}(i), and \ref{prop:perm2}(ii).

		    The methods of \cref{lem:isogenous} do not work to bound the sum over isogenous curves in this case since the coefficients grow exponentially in $T$.

    \subsection{Central Limit Theorems}
\subsubsection{Central Limit Theorem for elliptic curve pairs}
Let $E(a,b)$ denote the elliptic curve represented by the equation 
\begin{align*}
y^2 = x^3 + ax + b,
\end{align*}
and let $\Delta (a,b)$ be its discriminant. Let $a_{E(a,b)}(p)$ denote the trace of the Frobenius morphism of $E(a,b)/\F_p$ and $\widetilde{a}_{E(a,b)}(p) := a_{E(a,b)}(p) / \sqrt{p}$ the normalised trace.

Given two functions $f(n), g(n)$ on the integers, define the following family of pairs of elliptic curves 
\begin{align*}
T_{f,g}(A,B) := \{(E_1, E_2) \mid E_1 \not \simeq E_2, E_i= E(f(a_i),g(b_i)),  |a_i|\leq A, |b_i| \leq B, i = 1,2\}.
\end{align*}
Give any infinite set $P$ of primes, let $P_x :=\{p \leq x \mid p \in P\}$, and let $\pi_P (x) := \# P_x$. We say that a polynomial is a \textit{permutation polynomial for $P$} if it is a permutation polynomial for each prime in $P$.

We now prove Theorem \ref{pclt}(i), which is a Central Limit Theorem for terms of the form $\widetilde{a}_{E_1}(p) \widetilde{a}_{E_2}(p)$, where $E_1$ and $E_2$ are non-isomorphic elliptic curves.

\begin{proof}[Proof of Theorem \ref{pclt}(i)]
Let 
\begin{align*}
V_{x,r} := \frac{1}{\# T(A,B)}\sum_{(E,E')\in T (A,B)}
\left(\frac{\sum_{p \in P_x} \widetilde{a}_{E}(p)\widetilde{a}_{E'}(p)}{\sqrt{\pi_P(x)}}\right)^r.
\end{align*}
To prove the theorem, it is sufficient to determine the asymptotic behaviour of $V_{x,r}$ as $x \rightarrow \infty$ for all positive integers $r$, and in particular, show that 
\begin{align*}
V_{x,r} \rightarrow \frac{1}{\sqrt{2 \pi}}\int_{-\infty }^{ \infty}t^r e ^{-t^2 /2}dt,
\end{align*}
as $x \rightarrow \infty$.

First, we rewrite $V_{x,r}$ as 
\begin{align}\notag 
&\frac{1}{\# T(A,B)}
\frac{1}{\pi_P(x) ^{r/2}}
\sum_{(E,E')\in T (A,B)}
\sum_{w = 1 }^{ r} \sum_{r_1 + \dots + r_w = r} \frac{r!}{r_1! \dots r_w! w!}\\
&\cdot \sum_{\substack{p_1, \dots, p_w \in P_x \\ {p_i}'s \text{ all distinct }}}
 \left(\widetilde{a}_{E}(p_1)\widetilde{a}_{E'}(p_1)\right)^{r_1} \dots  \left(\widetilde{a}_{E}(p_w)\widetilde{a}_{E'}(p_w)\right)^{r_w} \notag \\
\notag 
 =& \frac{1}{\# T(A,B)}
\frac{1}{\pi_P(x) ^{r/2}}
\sum_{w = 1 }^{ r} \sum_{r_1 + \dots + r_w = r} \frac{r!}{r_1! \dots r_w! w!}
 \sum_{\substack{p_1, \dots, p_w \in P_x \\ {p_i}'s \text{ all distinct }}}  \sum_{(E,E')\in T (A,B)}\\ \label{eqn0}
& \left(\widetilde{a}_{E}(p_1)^{r_1}\dots \widetilde{a}_{E}(p_w)^{r_w} \right)  \left(\widetilde{a}_{E'}(p_1)^{r_1} \dots \widetilde{a}_{E'}(p_w)^{r_w} \right) .
\end{align}
and the inner sum becomes 
\begin{align}\label{sqdeqn}
&\left( \sum_{\substack{|a| \leq A \\  |b| \leq B \\ \Delta(f(a),g(b)) \neq 0 }} 
 \widetilde{a}_{E(f(a),g(b))}(p_1)^{r_1}\dots \widetilde{a}_{E(f(a),g(b))}(p_w)^{r_w} \right)^2 \\
&- \sum_{(E,E') \in T(A,B)}
 \left(\widetilde{a}_{E}(p_1)^{r_1}\dots \widetilde{a}_{E}(p_w)^{r_w} \right)   \left(\widetilde{a}_{E'}(p_1)^{r_1}\dots \widetilde{a}_{E'}(p_w)^{r_w} \right) \notag 
\end{align}
With regard to the second sum in equation (\ref{sqdeqn}), we note that 
\begin{align*}
& \frac{1}{\# T(A,B)} \frac{1}{\pi_P(x) ^{r/2}}
\sum_{w = 1 }^{ r} \sum_{r_1 + \dots + r_w = r} \frac{r!}{r_1! \dots r_w! w!}
 \sum_{\substack{p_1, \dots, p_w \in P_x \\ {p_i}'s \text{ all distinct }}}\\
& \sum_{(E,E') \in T(A,B)} 
 \left(\widetilde{a}_{E}(p_1)^{r_1}\dots \widetilde{a}_{E}(p_w)^{r_w} \right)  \left(\widetilde{a}_{E'}(p_1)^{r_1}\dots \widetilde{a}_{E'}(p_w)^{r_w} \right) \\
 \ll_r &\frac{1}{(4AB)^2}\frac{1}{\pi_P(x) ^{r/2}}  \pi_P(x) ^{r/2} \cdot \left(4AB \cdot 2 ^{2r}x^r\right),
\end{align*}
which tends to zero as $x \rightarrow \infty$ under the conditions $\log A/\log x ,\log B/\log x  \rightarrow \infty$.

Now we consider the first sum in equation (\ref{sqdeqn}).
From \cite{CDF97} we obtain 
\begin{align*}
\widetilde{a}_{E(f(a),g(b))}(p) ^m = \sum_{j=0 }^{ m}h_m (j) \widetilde{a}_{E(f(a),g(b))}(p^j) 
\end{align*}
with
\begin{align}
h_m (j):= \frac{2^{m + 1}}{\pi} \int_{0 }^{ \pi} \cos^m \theta \sin (j + 1)\theta \sin \theta d \theta. \label{heqn}
\end{align}
Note that $h_m (j)$ vanishes if the parity of $m$ and $j$ are different.

So 
\begin{align}
&\sum_{\substack{|a| \leq A \\  |b| \leq B \\ \Delta(f(a),g(b)) \neq 0 }} 
\widetilde{a}_{E(f(a),g(b))}(p_1)^{r_1} \dots \widetilde{a}_{E(f(a),g(b))}(p_w)^{r_w} \notag \\
&= \sum_{\substack{|a| \leq A \\  |b| \leq B \\ \Delta(f(a),g(b)) \neq 0 }} 
\sum_{j_1 = 0 }^{ r_1}h_{r_1}(j_1) \widetilde{a}_{E(f(a),g(b))}(p_1^{j_1})\dots  \sum_{j_w = 0 }^{ r_w}h_{r_w}(j_w) \widetilde{a}_{E(f(a),g(b))}(p_w^{j_w}) \notag \\
&= \sum_{j_1 }^{ r_1} \dots \sum_{j_w}^{ r_w}h_{r_1}(j_1) \dots h_{r_w}(j_w) \sum_{\substack{|a| \leq A \\  |b| \leq B \\ \Delta(f(a),g(b)) \neq 0 }} 
\widetilde{a}_{E(f(a),g(b))}(p_1^{j_1} \dots p_w^{j_w}) \label{eqn1}.
\end{align}

We now make use of Lemma 4.1 of \cite{BP19}.
\begin{lemma}\label{trfm}
Let $s(n)$ be the largest square-free number that divides a positive integer $n$, and let 
\begin{align*}
S (n):= \frac{1}{s (n) ^2} \sum_{\substack{a,b = 1 \\ (\Delta (a,b),n)=1}}^{ s (n)} \frac{{a}_{E(a,b)}(n)}{ \sqrt{n}}.
\end{align*}

Then for all $A, B \geq 1$ and any positive integer $n$,
\begin{align*}
\sum_{\substack{|a|\leq A,|b|\leq B \\ (\Delta (a,b),n)=1}} 
\widetilde{a}_{E(a,b)}(n) = 4AB S(n) + O\left(d(n) s(n)^2\right) + O\left(d(n)s(n)(A + B)\right).
\end{align*}
\end{lemma}

The proof of this Lemma can be adjusted to give, for $f$ and $g$ that are permutation polynomials with respect to an infinite set of primes $P$, that
\begin{align*}
\sum_{\substack{|a|\leq A,|b|\leq B \\ (\Delta (f(a),g(b)),n)=1}} \widetilde{a}_{E(f(a),g(b))}(n) = 4AB S(n) + O\left(d(n) s(n)^2\right) + O\left(d(n)s(n)(A + B)\right),
\end{align*}
for all $n$ that are products of primes in $P$. 

Applying this, we obtain that the innermost sum of equation (\ref{eqn1}) is
\begin{align}\notag 
\sum_{\substack{|a|\leq A,|b|\leq B \\ p_i \nmid \Delta (f(a),g(b))}} \widetilde{a}_{E(f(a),g(b))}(p_1^{j_1} \dots p_w^{j_w}) 
= &4AB S(p_1^{j_1} \dots p_w^{j_w}) + O\left(d(p_1^{j_1} \dots p_w^{j_w}) s(p_1^{j_1} \dots p_w^{j_w})^2\right) \\ \label{treqn}
&+ O\left(d(p_1^{j_1} \dots p_w^{j_w}) p_1 \dots p_w(A + B)\right).
\end{align}
Since $d (p_1^{j_1} \dots p_w^{j_w}) = \prod_{i=1 }^{ w}(j_i+1) \leq \prod_{i=1 }^{ w}(r_i+1) \leq 2 ^r$, the line above becomes
\begin{align*}
= 4AB S(p_1^{j_1} \dots p_w^{j_w}) + O_r\left((p_1 \dots p_w)^2 +  p_1 \dots p_w(A + B)\right).
\end{align*}

\textbf{Case 1:} $r_\ell$ is odd for some $\ell$.\\

If $j_\ell$ is even, then $h_{r_\ell}(j_\ell)=0$, given equation (\ref{heqn}), so equation (\ref{eqn1}) vanishes.

Now let us assume that $j_\ell$ is odd. Since $S (n)$ is multiplicative \cite{BP19}, 
we have 
\begin{align*}
S(p_1^{j_1} \dots p_w^{j_w}) = \prod_{i = 1 }^{ w} S(p_i^{j_i}) = 0,
\end{align*}
the last equality holding since $S(p^{j_\ell}) = 0$, using equation (16) of \cite{MP20} (derived from Lemma 8.2 of \cite{BP19}). 
The right-hand side of equation (\ref{treqn}) becomes 
\begin{align*}
O_r\left((p_1 \dots p_w)^2+p_1 \dots p_w(A + B)\right),
\end{align*}
and (in the context of equation (\ref{eqn0})) we note that 
\begin{align*}
&\frac{1}{\# T(A,B)} \frac{1}{\pi_P(x) ^{r/2}}
\sum_{(E,E') \in T(A,B)}
 \sum_{\substack{p_1, \dots, p_w \in P_x \\ {p_i}'s \text{ all distinct }}}
 \left(\widetilde{a}_{E}(p_1)\widetilde{a}_{E'}(p_1)\right)^{r_1} \dots  \left(\widetilde{a}_{E}(p_w)\widetilde{a}_{E'}(p_w)\right)^{r_w} \\
&\ll_r \frac{1}{\# T(A,B)} \frac{1}{\pi_P(x) ^{r/2}}
\sum_{\substack{p_1, \dots, p_w \in P_x \\ {p_i}'s \text{ all distinct }}}(p_1 \dots p_w)^{4}(A + B)^2\\
&\ll_r \frac{(A+B)^2}{(AB)^2}  \frac{\pi_P(x)^r(x^r)^4 }{\pi_P(x) ^{r/2}} \rightarrow 0,
\end{align*}
 as $x \rightarrow \infty$, under the conditions that $\log A / \log x$ and $\log B / \log x$ tend to infinity as $x \rightarrow \infty$.

\textbf{Case 2:} $(r_1, \dots, r_w) = (2,\dots,2)$.
Since $h_2(j)=1,0,1$ for $j = 0,1,2$, respectively, equation (\ref{eqn1}) becomes 
\begin{align*}
&\sum_{j_1=0,2 } \dots \sum_{j_w=0,2} \sum_{\substack{|a| \leq A \\  |b| \leq B \\ \Delta(f(a),g(b)) \neq 0 }} 
\widetilde{a}_{E(f(a),g(b))}(p_1^{j_1} \dots p_w^{j_w}) = \sum_{\substack{|a| \leq A \\  |b| \leq B \\ \Delta(f(a),g(b)) \neq 0 }} 
1 \\
&+ \sum_{\substack{j_1, \dots, j_w=0,2 \\ (j_1,\dots,j_w)\neq (0,\dots,0)}} 
 \left[4AB S(p_1^{j_1} \dots p_w^{j_w}) + O_r\left((p_1 \dots p_w)^2 +  p_1 \dots p_w(A + B)\right)\right].
\end{align*}
Therefore, we have 
\begin{align*}
\frac{1}{\# T(A,B)} &\frac{1}{\pi_P(x) ^{r/2}}
 \sum_{\substack{p_1, \dots, p_w \in P_x \\ {p_i}'s \text{ all distinct }}}
\left( \sum_{\substack{|a| \leq A \\  |b| \leq B \\ \Delta(f(a),g(b)) \neq 0 }} 
 \left(\widetilde{a}_{E(f(a),g(b))}(p_1)^{r_1} \dots \widetilde{a}_{E(f(a),g(b))}(p_w)^{r_w} \right) \right)^2 \\
 \sim   1+  & \frac{1}{\# T(A,B)} \frac{1}{\pi_P(x) ^{r/2}}
 \sum_{\substack{j_1, \dots, j_w=0,2 \\ (j_1,\dots,j_w)\neq (0,\dots,0)}}  \sum_{\substack{p_1, \dots, p_w \in P_x \\ {p_i}'s \text{ all distinct }}} \\
& \left[32(AB)^2S(n) + (8AB) O_r\left(s(n)^2 + s(n)  (A + B) \right) + (4AB)^2 S(n)^2 \right. \\
& \left.  + O_r\left(AB S(n)s(n)^2\right) + O_r \left(AB S(n) s(n)(A + B)\right)+
O_r\left(s(n)^2 + s(n)(A + B)  \right)^2
\right]
\end{align*}
where $n$ denotes $p_1^{j_1} \dots p_w^{j_w}$. We claim that this expression tends to 1 under the conditions of Theorem \ref{pclt}(i). We will explicitly demonstrate the asymptotic behaviour for two of the terms in the expression above; the rest follow in a similar manner.

In the case of 
\begin{align} \label{term1}
\frac{1}{\# T(A,B)} \frac{1}{\pi_P(x) ^{r/2}}
 \sum_{\substack{j_1, \dots, j_w=0,2 \\ (j_1,\dots,j_w)\neq (0,\dots,0)}}  \sum_{\substack{p_1, \dots, p_w \in P_x \\ {p_i}'s \text{ all distinct }}} 
32 (AB)^2 S(n)
\end{align}
we note that equation (67) of \cite{BP19} states that $S(p^k) \ll k/ \sqrt{p}$ and
from equation (27) of \cite{Na06} we get 
\begin{align*}
\sum_{\substack{j_1, \dots, j_w=0,2 \\ (j_1,\dots,j_w)\neq (0,\dots,0)}}  \sum_{\substack{p_1, \dots, p_w \in P_x \\ {p_i}'s \text{ all distinct }}} 
\frac{1}{\sqrt{p_1^{j_1} \dots p_w^{j_w}}}
\ll_r \pi_P(x) ^{r/2-1} \left(\sum_{p \in P_x}1/p \right)^{r/2}
\end{align*}
Therefore, we majorize expression (\ref{term1}) by
\begin{align*}
\ll_r  \frac{ \left(\sum_{p \in P_x}1/p \right)^{r/2}}{\pi_P(x)} \rightarrow 0,
\end{align*}
as $x \rightarrow \infty$.

Let us also consider the case of 
\begin{align*} 
&\frac{1}{\# T(A,B)} \frac{1}{\pi_P(x) ^{r/2}}
 \sum_{\substack{j_1, \dots, j_w=0,2 \\ (j_1,\dots,j_w)\neq (0,\dots,0)}}  \sum_{\substack{p_1, \dots, p_w \in P_x \\ {p_i}'s \text{ all distinct }}} 
8AB \cdot  O_r\left(s(n) (A+B)\right)\\
\ll_r &\frac{1}{{\rm min}(A,B)}\frac{1}{\pi_P(x) ^{r/2}} \sum_{\substack{p_1, \dots, p_w \in P_x}} p_1 \dots p_w \\
\ll_r &\frac{1}{{\rm min}(A,B)}\frac{1}{\pi_P(x) ^{r/2}} x^{r/2} \pi_P(x) ^{r/2}
\end{align*}
which tends to zero since $\log A / \log x$ and $\log B / \log x$ tend to $\infty$ as $x \rightarrow \infty$.

\textbf{Case 3:} The $r_i$ values are all even, with $(r_1, \dots, r_w) \neq (2,\dots,2)$.

The inside sum of equation (\ref{eqn0}) can be bounded above by $(4AB)^2 (2^r)^2$, and so 
\begin{align*}
&\frac{1}{\# T(A,B)} \frac{1}{\pi_P(x) ^{r/2}}
\sum_{\substack{p_1, \dots, p_w \in P_x \\ {p_i}'s \text{ all distinct }}} \\
&\sum_{(E,E') \in T(A,B)}
 \left(\widetilde{a}_{E}(p_1)^{r_1} \dots \widetilde{a}_{E}(p_w)^{r_w} \right)  \left(\widetilde{a}_{E'}(p_1)^{r_1} \dots \widetilde{a}_{E'}(p_w)^{r_w} \right) \\
& \ll \frac{1}{\# T(A,B)} \frac{1}{\pi_P(x) ^{r/2}} \pi_P(x) ^w (4AB)^2 (2^r)^2 \\
& \ll_r \frac{1}{\pi(x)} \rightarrow 0,
\end{align*}
as $x \rightarrow \infty$, since $w \leq r/2 -1$. This finishes case 3.\\

Now we combine the results from the three cases to obtain that
\begin{align*}
V_{x,r}  = 
\begin{cases} 
\frac{r!}{2 ^{r/2}(r/2)!}, \text{ for even } r , \\ 
0, \text{ otherwise } .
\end{cases}
\end{align*}
Since the same holds true for the integral 
\begin{align*}
\frac{1}{\sqrt{2 \pi}}\int_{-\infty }^{ \infty}t ^r e ^{-t ^2 /2}dt,
\end{align*}
this completes the proof of the theorem.
\end{proof}

\begin{proof}[Proof of Theorem \ref{eclt}(ii)]
The proof starts in the same way as that of the previous theorem. Note that $\widetilde{a}_{E(f(a),g(b))} (n)$ is still periodic in $a,b$ but now with period $\gamma (n) := \varphi (s(n)^2)$. As a result, the formal derivation of a suitable trace formula analogous to Lemma \ref{trfm} gives
\begin{align*}
\sum_{\substack{|a|\leq A,|b|\leq B \\ (\Delta (f(a),g(b)),n)=1}} 
\widetilde{a}_{E(f(a),g(b))}(n) = 4AB S_\gamma(n) + O\left(d(n) \gamma(n)^2\right) + O\left(d(n)\gamma(n)(A + B)\right).
\end{align*}
where
\begin{align*}
S_\gamma (n) := \frac{1}{\gamma (n) ^2} \sum_{\substack{a,b = 1 \\ (\Delta (a,b),n)=1}}^{ \gamma (n)} \frac{{a}_{E(f(a),g(b))}(n)}{ \sqrt{n}}.
\end{align*}
However, one can observe that $S_\gamma(n)= S(n)$ so we get
\begin{align*}
\sum_{\substack{|a|\leq A,|b|\leq B \\ (\Delta (f(a),g(b)),n)=1}} 
\widetilde{a}_{E(f(a),g(b))}(n) = 4AB S(n) + O\left(d(n) s(n)^4\right) + O\left(d(n)s(n)^2 (A + B)\right).
\end{align*}
The rest of the proof proceeds as before.\\
\end{proof}

\subsubsection{Modular forms} One can adjust the proof above to obtain a similar statement for modular forms by using a trace formula. Let $\mathcal{F}_{k,N}$ be the set of normalised cuspidal Hecke eigenforms of weight $k$ and level $N$, and let $\lambda_f (n)$ denote the Hecke eigenvalue of $f \in \mathcal{F}_{k,N}$ at $n$. Then from Section 4 of \cite{Se97} we have 
\begin{align*}
\sum_{f \in \mathcal{F}_{k,N}} \lambda_f (n) = \delta_n \frac{k-1}{12} \frac{1}{\sqrt{n}}\psi(N) + O\left(n ^{1/2}N^{1/2}d(N)\right),
\end{align*}
where $\psi(N)=[SL_2(\Z):\Gamma_0(N)]=N \prod_{p \mid N}(1 + 1/p)$, and where $\delta_n$ is 1 if $n$ is a square, and 0 otherwise.
Then, following the same approach as in the previous proof, we establish Theorem \ref{mf-clt}.

\section{Restricted Hurwitz Number Sums\label{sec:res}}
	\subsection{Character Sums}
		\begin{defn}
			A multiplicative function is a function $f:\N^n\to \Q$ is such that
			$$f(x_1,\dots,x_n)=\prod_{p\mid x_1x_2\dots x_n}f\la(p^{v_p(x_1)},\dots,p^{v_p(x_n)}\ra).$$
		\end{defn}
		\begin{lem}\label{lem:euler}
			For any multiplicative function $f:\N^n\to \Q$ such that $\sum_{(x_1,\dots,x_n)\in\N^n}f(x_1,\dots,x_n)$ converges absolutely we have the absolutely convergent Euler product
			$$\sum_{(x_1,\dots,x_n)\in\N^n}f(x_1,\dots,x_n)=\prod_{p}\sum_{(i_1,\dots,i_n)\in \Z_{\geq 0}^n}f(p^{i_1},\dots,p^{i_n}).$$
		\end{lem}
		\begin{defn}
			For $f,g,m,n\geq 1$, $(\congval,\modval)=1$, and $2\nmid \tau$ let
			\begin{align*}
				c_{f,g}^{\tau,\congval,\modval}(m,n)	&=\kro{f^2g^2}{2\tau}\sum_{\substack{a(4m)^*\\(\tau^2-af^2,4m)=4\\\frac{\tau^2-af^2}{4}\equiv \congval\mod (\modval,mf^2)}}\sum_{\substack{b(4n)^*\\(\tau^2-bg^2,4n)=4\\\frac{\tau^2-af^2}{4}\equiv \frac{\tau^2-bg^2}{4} \mod (mf^2,ng^2)\\\frac{\tau^2-bg^2}{4}\equiv \congval\mod (\modval,ng^2)}}\kro{a}{m}\kro{b}{n},\\
				c_{f}^{\tau,\congval,\modval}(m)	&=c_{f,1}^{\tau,\congval,\modval}(m,1)=\kro{f^2}{2\tau}\sum_{\substack{a(4m)^*\\(\tau^2-af^2,4m)=4\\\frac{\tau^2-af^2}{4}\equiv \congval\mod (\modval,mf^2)}}\kro{a}{m}.
			\end{align*}
		    where $\kro{\cdot}{\cdot}$ is the Kronecker symbol.
		\end{defn}
		\begin{lem}\label{lem:cstar2}
			Let $i,j,k,\ell\geq 0$.
			\begin{thmlist}
				\item\label{lem:cstar2mult} 	
					$c_{f,g}^{\tau,\congval,\modval}(m,n)$ is multiplicative in the four variables $f,g,m,n$
				\item\label{lem:cstar21eval}
					We have
					\begin{align*}
						c_{1,1}^{\tau,\congval,\modval}(p^i,p^j)
						&=\pw{
								1												&	i=j=0,\\
								2^{\max(i,j)-1}(-1)^{i+j}									&	p=2,v_p(\modval)=0,i+j>0,\\
								-p^{\max(i,j)-1}\kro{\tau^2}{p}								&	v_p(\modval)=0,2\nmid i+j>0,\\
								p^{\max(i,j)-1}\la[p-1-\kro{\tau^2}{p}\ra]						&	v_p(\modval)=0,2\mid i+j>0,\\
								p^{\max(i-v_p(\modval),j-v_p(\modval),0)}\kro{\tau^2-4\congval}{p}^{i+j}	&	v_p(\modval)>0.
							}
					\end{align*}
				\item\label{lem:cstar2peval}
					Since $c_{p^k,p^\ell}^{\tau,\congval,\modval}(p^i,p^j)=c_{p^\ell,p^k}^{\tau,\congval,\modval}(p^j,p^i)$ suppose $k>\ell$.
					If $v_p(2\tau)\geq 1$ or $\max(v_p(\modval),2k)>v_p(\tau^2-4\congval)$
					then
					$$c_{p^k,p^\ell}^{\tau,\congval,\modval}(p^i,p^j)=c_{p^k,p^k}^{\tau,\congval,\modval}(p^i,p^j)=0.$$
					If $v_p(2\tau)=0$ and $\max(v_p(\modval),2k)\leq v_p(\tau^2-4\congval)$ we have
					\begin{align*}
						c_{p^k,p^\ell}^{\tau,\congval,\modval}(p^i,p^j)	
						&=\pw{	0								&	j>0,\\	
								1								&	i=j=0,\\
								0								&	2\nmid i>j=0,2k\geq v_p(\modval),\\	
								p^{i-1}(p-1)							&	2\mid i>j=0,2k\geq v_p(\modval),\\	
								p^{\max(i+2k-v_p(\modval),0)}\kro{\rho_k}{p}^i	&	i>j=0,2k<v_p(\modval),
							}\\
						c_{p^k,p^k}^{\tau,\congval,\modval}(p^i,p^j)
						&=\pw{
								1												&	i+j=0,\\
								0												&	2\nmid i+j>0,2k\geq v_p(\modval),\\
								p^{\max(i,j)-1}(p-1)									&	2\mid i+j>0,2k\geq v_p(\modval),\\
								p^{\max(i+2k-v_p(\modval),j+2k-v_p(\modval),0)}\kro{\rho_k}{p}^{i+j}	&	2k< v_p(\modval).
							}
					\end{align*}
					where $\rho_k=p^{-2k}(\tau^2-4\congval)$.
				\item\label{lem:cstar2bound} 	
					We have
					$$\abs{c_{f,g}^{\tau,\congval,\modval}(m,n)} \ll \frac{mn}{\kappa(mn)(m,n)}$$
					where $\kappa$ the multiplicative function generated by
					$$\kappa(p^i)=\pw{p	&	2\nmid i,\\1	&	2\mid i.}$$
			\end{thmlist}
			\begin{proof}
				(i) 
				We have $c_{1,1}^{\tau,\congval,\modval}(1,1)=1$.
				Now suppose $(f_1g_1m_1n_1,f_2g_2m_2n_2)=1$.
				If $(f_1f_2g_1g_2,2\tau)\neq 1$ then
				$$c_{f_1f_2,g_1g_2}^{\tau,\congval,\modval}(m_1m_2,n_1n_2)=0=c_{f_1,g_1}^{\tau,\congval,\modval}(m_1,n_1)c_{f_2,g_2}^{\tau,\congval,\modval}(m_2,n_2),$$
				so suppose $(f_1f_2g_1g_2,2\tau)= 1$.
				Note that $(\tau^2-af_1^2f_2^2,4m_1m_2)=4$ implies $a\equiv 1\pmod 4$.
				Therefore, by the Generalized Chinese Remainder Theorem, for invertible $a\pmod{4m_1m_2}$ we have $(\tau^2-af_1^2f_2^2,4m_1m_2)=4$ if and only if there exist invertible $a_1\equiv a\pmod {4m_1}$ and $a_2\equiv a\pmod {4m_2}$ such that $(\tau^2-a_1f_1^2f_2^2,4m_1)=4$ and $(\tau^2-a_2f_1^2f_2^2,4m_2)=4$.
				The residues $b\pmod {4n_1n_2}$ are split up similarly.
				Furthermore we have $\kro{a}{m_1m_2}=\kro{a_1}{m_1}\kro{a_2}{m_2}$, $\kro{b}{n_1n_2}=\kro{b_1}{n_1}\kro{b_2}{n_2}$, and the congruences
				\begin{align*}
					\frac{\tau^2-af_1^2f_2^2}{4}	&\equiv	\frac{\tau^2-bg_1^2g_2^2}{4}	\mod (m_1m_2f_1^2f_2^2,n_1n_2g_1^2g_2^2),\\
					\frac{\tau^2-af_1^2f_2^2}{4}	&\equiv	\congval	\mod (\modval,m_1m_2f_1^2f_2^2),\\
					\frac{\tau^2-bg_1^2g_2^2}{4}	&\equiv 	\congval	\mod (\modval,n_1n_2g_1^2g_2^2),
				\end{align*}
				are satisfied if and only if the congruences
				\begin{align*}
					\frac{\tau^2-a_if_1^2f_2^2}{4}	&\equiv	\frac{\tau^2-b_ig_1^2g_2^2}{4}	\mod (m_if_i^2,n_ig_i^2),\\
					\frac{\tau^2-a_if_1^2f_2^2}{4}	&\equiv	\congval	\mod (\modval,m_if_i^2),\\
					\frac{\tau^2-b_ig_1^2g_2^2}{4}	&\equiv 	\congval	\mod (\modval,n_ig_i^2),
				\end{align*}
				are all satisfied for $i=1,2$.
				Finally, since $(f_2,4m_1)=1$ we have $f_2^2\mod 4m_1$ is invertible so we can instead sum over invertible residues $f_2^{-2}a_1\pmod {4m_1}$ which cancels the factors of $f_2$ in the above congruences.
				Furthermore $\kro{a_1}{m_1}=\kro{f_2^{-2}a_1}{m_1}$.
				Similar relations hold for $a_2,b_1,b_2$, resulting in
				$$c_{f_1f_2,g_1g_2}^{\tau,\congval,\modval}(m_1m_2,n_1n_2)=c_{f_1,g_1}^{\tau,\congval,\modval}(m_1,n_1)c_{f_2,g_2}^{\tau,\congval,\modval}(m_2,n_2).$$
				Thus $c_{f,g}^{\tau,\congval,\modval}(m,n)$ is multiplicative in the four variables $f,g,m,n$.

				(ii)
				The case $i=j=0$ is trivial.

				If $v_p(\modval)=0$ then $(v,p^i)=1=(v,p^j)$, meaning the conditions
				\begin{align*}
					\frac{\tau^2-a}{4}	&\equiv	\congval	\mod (\modval,p^i),\\
					\frac{\tau^2-b}{4}	&\equiv 	\congval	\mod (\modval,p^j),
				\end{align*}
				are trivial given that the conditions $(\tau^2-a,4p^i)=4$ and $(\tau^2-b,4p^j)=4$ imply $a\equiv b\equiv 1\pmod 4$.
				Therefore $c_{1,1}^{\tau,\congval,\modval}(p^i,p^j)=c_{1,1}^{\tau}(p^i,p^j)$ which is evaluated in \cite[Lemma 4.7(i)]{adj}.

				The proof $v_p(\modval)>0$ is the same as the proof of the case $v_p(\modval)>2k$ in part (iii).

				(iii)
				Again the cases $i=j=0$ are trivial.

				For the case $j>0$, the condition
				$$\frac{\tau^2-ap^{2k}}{4}\equiv \frac{\tau^2-bp^{2\ell}}{4}\mod (p^{i+2k},p^{i+2\ell})$$
				is equivalent to
				$$ap^{2k}\equiv bp^{2\ell}\mod (p^{\max(i+2k,j+2\ell)}),$$
				which is impossible because $v_p(ap^{2k})=2k>2\ell=v_p(bp^{2\ell})$ and $\max(i+2k,j+2\ell)>2\ell=v_p(bp^{2\ell})$.

				If $2k\geq v_p(\modval)$ we have $(\modval,p^{i+2k})=(\modval,p^{2k})$.
				Therefore, given that $(\tau^2-ap^{2k},4p^i)=4$, the condition
				$$\frac{\tau^2-a}{4}	\equiv	\congval	\mod (\modval,p^{i+2k})$$
				is equivalent to $\max(v_p(\modval),2k)\leq v_p(\tau^2-4\congval)$.
				Thus
				\begin{align*}
					c_{p^k,p^{\ell}}^{\tau,\congval,\modval}(p^i,1)	&=c_{p^k,p^{\ell}}^{\tau}(p^i,1),\\
					c_{p^k,p^{k}}^{\tau,\congval,\modval}(p^i,p^j)	&=c_{p^k,p^{k}}^{\tau}(p^i,p^j)
				\end{align*}
				which are evaluated in \cite[Lemma 4.5-7]{adj}.
				
				If $v_p(\modval)>2k$ we have
				\spanner{\begin{align*}
					c_{p^k,p^\ell}^{\tau,\congval,\modval}(p^i,1)	
						&=\sum_{\substack{a(4p^i)^*\\(\tau^2-ap^{2k},4m)=4\\\frac{\tau^2-ap^{2k}}{4}\equiv \congval\mod (\modval,p^{i+2k})}}\sum_{\substack{b(4)^*\\(\tau^2-bp^{2\ell},4)=4\\\frac{\tau^2-ap^{2k}}{4}\equiv \frac{\tau^2-bp^{2\ell}}{4} \mod (p^{i+2k},p^{2\ell})\\\frac{\tau^2-bp^{2\ell}}{4}\equiv \congval\mod (\modval,p^{2\ell})}}\kro{a}{p}^i\kro{b}{p}^0\\
					&=\sum_{\substack{a(4p^i)^*\\a\equiv p^{-2k}(\tau^2-4\congval)\pmod{4p^{\min(v_p(\modval)-2k,i)}}}}\kro{a}{p}^i.\\
					&=p^{\max(i+2k-v_p(\modval),0)} \kro{\rho_k}{p}^i,
				\end{align*}}
				Suppose $i\geq j$.
				We have
				\spanner{\begin{align*}
					c_{p^k,p^k}^{\tau,\congval,\modval}(p^i,p^j)	
						&=\sum_{\substack{a(4p^i)^*\\(\tau^2-ap^{2k},4m)=4\\\frac{\tau^2-ap^{2k}}{4}\equiv \congval\mod (\modval,p^{i+2k})}}\sum_{\substack{b(4p^j)^*\\(\tau^2-bp^{2k},4p^j)=4\\\frac{\tau^2-ap^{2k}}{4}\equiv \frac{\tau^2-bp^{2k}}{4} \mod (p^{i+2k},p^{j+2k})\\\frac{\tau^2-bp^{2k}}{4}\equiv \congval\mod (\modval,p^{j+2k})}}\kro{a}{p}^i\kro{b}{p}^j\\
						&=\sum_{\substack{a(4p^i)^*,b(4p^j)^*\\a\equiv p^{2k}\tau^2-4\congval\pmod{4p^{\min(v_p(\modval)-2k,i)}}\\a\equiv b\pmod {p^{\min(i,j)}}\\b\equiv p^{-2k}\tau^2-4\congval\pmod {4p^{\min(v_p(\modval)-2k,j)}}}}\kro{a}{p}^i\kro{b}{p}^j\\
						&=\sum_{\substack{a(4p^i)^*\\a\equiv p^{2k}\tau^2-4\congval\pmod{4p^{\min(v_p(\modval)-2k,i)}}}}\kro{a}{p}^{i+j}\\
						&=p^{\max(i+2k-v_p(\modval),0)} \kro{\rho_k}{p}^{i+j}.
				\end{align*}}

				(iv)
				By the above we have

				\spanner{\begin{align*}
					\abs{c_{f,g}^{\tau,\congval,\modval}(m,n)}	&=\la[\prod_{\substack{p\mid fgmn\\p\mid \modval}}\abs{c_{p^{v_p(f)},p^{v_p(g)}}^{\tau,\congval,\modval}(p^{v_p(m)},p^{v_p(n)})}\ra]\la[\prod_{\substack{p\mid fgmn\\p\mid \modval}}\abs{c_{p^{v_p(f)},p^{v_p(g)}}^{\tau,\congval,\modval}(p^{v_p(m)},p^{v_p(n)})}\ra]\\
										&\ll\la[\prod_{\substack{p\mid fgmn\\p\mid \modval}}\frac{p^{v_p(m)+v_p(n)+1}}{\kappa(p^{v_p(m)+v_p(n)})(p^{v_p(m)},p^{v_p(n)})}\ra]\la[\prod_{\substack{p\mid fgmn\\p\nmid \modval}}\frac{p^{v_p(m)+v_p(n)}}{\kappa(p^{v_p(m)+v_p(n)})(p^{v_p(m)},p^{v_p(n)})}\ra]\\
										&\leq \frac{\modval mn}{\kappa(mn)(m,n)}
										\ll \frac{mn}{\kappa(mn)(m,n)}.\qedhere
				\end{align*}}
			\end{proof}
		\end{lem}
	
	\subsection{Two-Curve Hurwitz Number Sum}
		\begin{prop}
			For odd $\tau$ and $(\congval,\modval)=1$ we have
			$$\sum_{\substack{p\leq x\\p\equiv \congval\pmod \modval}}\frac{\H{\tau}{p}^2}{p^2}=C_{\tau,\tau,\congval,\modval} \log \log x+O(1)$$
			where
			\begin{align*}
				C_{\tau,\tau,\congval,\modval}	&=\frac{4}{\pi^2}\prod_{p}\Lambda_{\tau,\tau,\congval,\modval}(p)>0\\
				\Lambda_{\tau,\tau,\congval,\modval}(p)
						&=	\pw{
								\frac{4}{9}																													& v_p(\modval)=0,p=2,\\
								\frac{p^2 (p^2+1)}{(p^2-1)^2}																									& v_p(\modval)=0,p\mid \tau,\\
								\frac{p^2 (p^4-2p^2-3p-1)}{(p+1)^3(p-1)^3}																							& v_p(\modval)=0,p\nmid 2\tau,\\
								\frac{\sigma_{-1}\la(p^{\ceil{v_p(\modval)/2}-1}\ra)^2}{\phi(p^{v_p(\modval)})}+\frac{2p^{\ceil{v_p(\modval)/2}}(p+1)^2-p^2-3p-1}{p^{4\ceil{v_p(\modval)/2}-4}(p^2-1)^3}			& 1\leq v_p(\modval)\leq v_p(\rho_0),\\
								\frac{1}{\phi(p^{v_p(\modval)})}																									& 0=v_p(\rho_0)<v_p(\modval),\\
								\frac{\sigma_{-1}\la(p^{v_p(\rho_0)/2-1/2}\ra)^2}{\phi(p^{v_p(\modval)})}																		& 0<v_p(\rho_0)<v_p(\modval),2\nmid v_p(\rho_0),\\
								\frac{1}{\phi(p^{v_p(\modval)})}\la[\sigma_{-1}\la(p^{v_p(\rho_0)/2}\ra)+\frac{1}{p^{v_p(\rho_0)/2}\la[\kro{\rho_*}{p}p-1\ra]}\ra]^2									& 0<v_p(\rho_0)<v_p(\modval),2\mid v_p(\rho_0),
							}
			\end{align*}
			with $\rho_0=\tau^2-4\congval$, $\rho_*=p^{-v_p(\tau^2-4\congval)}(\tau^2-4\congval)$, and $\sigma_\alpha$ is the sum-of-divisors function $\sigma_\alpha(n)=\sum_{d\mid n}d^\alpha$.
			\begin{proof}
				Similar to \cite{adj} we begin by considering
				\begin{equation*}
					\sum_{\substack{f,g\leq 2\sqrt{x}\\(fg,2\tau)=1}}\frac{1}{fg}\sum_{\substack{p\in S_{f,g}(x)\\p\equiv \congval\pmod \modval}}L(1,\chi_{d_1})L(1,\chi_{d_2})\log p,
				\end{equation*}
				where
				\begin{equation*}
					S_{f,g}(x)	=\set{\max(3,\tau^2/4)<p\leq x:d_1f^2=d_2g^2=\tau^2-4p}.
				\end{equation*}
				Following \cite{adj} up to \cite[Equation (14)]{adj}, for $U,\ep>0$, $0<\delta<\frac 12$, and $1<V\leq 2\sqrt x$ we have
				\begin{align}
					\sum_{\substack{f,g\leq 2\sqrt{x}\\(fg,2\tau)=1}}&\frac{1}{fg}\sum_{\substack{p\in S_{f,g}(x)\\p\equiv \congval\pmod \modval}}L(1,\chi_{d_1})L(1,\chi_{d_2})\log p\nonumber\\
					&=
					\sum_{\substack{f,g\leq V\\(fg,2\tau)=1}}\frac{1}{fg}
					\sum_{mn\leq U\log U}\frac{e^{-mn/U}}{mn}
					\sum_{\substack{a\pmod{4m}\\b\pmod{4n}}}\kro{a}{m}\kro{b}{n}
					\sum_{\substack{\max(3,\tau^2/4)<p\leq x\\p\nmid mn\\p\equiv \frac{\tau^2-af^2}{4}\pmod{mf^2}\\p\equiv \frac{\tau^2-bg^2}{4}\pmod{ng^2}\\p\equiv \congval\pmod \modval}}\log p\nonumber\\
					&\quad+O\la(\frac{x^{11/8+2\ep}}{U^{1/2}}+U^{\delta}\log U\log^2 V+\frac{x\log x\log^2 U}{V}+\frac{x\log x \log V^2}{(U\log U)^{1-\ep}}\ra).\label{eq:introcstar2}
				\end{align}
				Note that the error term $O\la(\frac{x\log x\log^2 U}{V}\ra)$ is slightly weaker than the term $O\la(\frac{x\log x\log^2 U}{V^2}\ra)$ found in \cite{adj}.
				This accounts for the error terms contributed by the sum over $f\leq V<g\leq \sqrt x$ and $g\leq V<f\leq \sqrt x$ that \cite{adj} did not consider.
				There is also an error term $O(U^{\delta}\log U\log^2 V)$ that both \cite{adj} and \cite{dp} did not consider, which comes from the primes $p\mid mn$.
				This comes from the bounds $\sum_{p\mid \ell}1\ll \log \ell\ll \ell^{\delta}$ and $\sum_{\ell=1}^{\infty}\frac{d(\ell)}{\ell^{1-\delta}}e^{-\ell/U}\ll U^{\delta}\log U$ obtained from moving the line of integration in the identity
				$$\sum_{\ell=1}^{\infty}\frac{d(\ell)}{\ell^{1-\delta}}e^{-\ell/U}=\frac{1}{2\pi i}\int_{(1)}\zeta(s+1-\delta)^2\Gamma(s+1)\frac{U^s}{s}\differ s,$$
				from $(1)$ to $(-1/2)$, in a similar manner to \cite[Lemma 3.2]{adj}.

				For $(a,n)=1$ we use the same notation as \cite{adj},
				$$\psi(x;n,a)=\sum_{\substack{p\leq x\\p\equiv a\pmod n}}\log p=\frac{x}{\phi(n)}+E(x;n,a).$$
				Then the first term in Equation \eqref{eq:introcstar2} is
				\begin{align*}
					x\sum_{\substack{f,g\leq V}}\frac{1}{fg}
						&\sum_{mn\leq U\log U}\frac{e^{-mn/U}}{mn}\frac{c_{f,g}^{\tau,\congval,\modval}(m,n)}{\phi([\modval,mf^2,ng^2])}\\
					&+O\la(\sum_{\substack{f,g\leq V\\mn\leq U\log U}}\frac{1}{fgmn}
						\sum\nolimits^*_{\substack{a\pmod{4m}\\b\pmod{4n}}}\abs{E(x,[\modval,mf^2,ng^2],\theta^*)}\ra)\\
					&+O\la(U\log^2 U\log^2 V\ra)
				\end{align*}
				where $\sum^*_{\substack{a\pmod{4m}\\b\pmod{4n}}}$ means the sum runs over invertible residues such that
				\begin{align*}
					\frac{\tau^2-af^2}{4}	&=\frac{\tau^2-bg^2}{4}	\mod (mf^2,ng^2)\\
					\frac{\tau^2-af^2}{4}	&=\congval	\mod (\modval,mf^2)\\
					\frac{\tau^2-bg^2}{4}	&=\congval	\mod (\modval,ng^2),
				\end{align*}
				and the unique residue $\theta^*\mod [\modval,mf^2,ng^2]$ such that
				\begin{align*}
					\theta^*	&=\frac{\tau^2-af^2}{4} \pmod{mf^2},\\
					\theta^*	&=\frac{\tau^2-bg^2}{4} \pmod{ng^2},\\
					\theta^*	&=\congval \pmod{\modval},
				\end{align*}
				is invertible.
				Note that the condition $(fg,2\tau)=1$ is now redundant due to the factor of $\kro{f^2g^2}{2\tau}$ in the definition of $c_{f,g}^{\tau,\congval,\modval}(m,n)$.
				The error term $O\la(U\log^2 U\log^2 V\ra)$ comes from the primes $p\leq \max(3,\tau^2/4)$.

				Let $c^*(\ell)$ be the number of ways to write $\ell=[n_1,n_2,n_3]$.
				Similar to the proof of \cite[Lemma 3.4]{adj} we have $c^*(\ell)\leq d(\ell)^3$.
				As stated in \cite{adj} we have
				$$\sum_{\ell\leq N}d^r(\ell)\ll N\log^{2^r-1}N$$
				so by partial summation we have $\sum_{\ell\leq N}\frac{d^6(\ell)}{\ell}\ll \log^{64}N$.
				By a theorem of Barban-Davenport-Halberstam \cite{dav}
				\spanner{\begin{align*}
					\sum_{\substack{f,g\leq V}}\frac{1}{fg}\sum_{mn\leq U\log U}&\frac{1}{mn}
						\sum\nolimits^*_{\substack{a\pmod{4m}\\b\pmod{4n}}}\abs{E(x,[\modval,mf^2,ng^2],\theta^*)}\\
					&\leq \modval\sum_{f,g\leq V}fg\sum_{\ell\leq \modval f^2g^2U\log U}\frac{1}{\ell}
						\sum_{\theta^*\pmod{\ell}^*}c^*(\ell)\abs{E(x,\ell,\theta^*)}\\
					&\leq \modval\sum_{f,g\leq V}fg
						\la(\sum_{\substack{\ell\leq \modval V^4 U\log U\\\theta^*\pmod \ell}}\frac{d^6(\ell)}{\ell^2}\ra)^{1/2}\la(\sum_{\substack{\ell\leq \modval V^4 U\log U\\\theta^*\pmod \ell}}E^2(x,\ell,\theta^*)\ra)^{1/2}\\
					&\ll \sum_{f,g\leq V}fg
						\log^{32}(V^4 U\log U)(V^4Ux \log U \log x)^{1/2}\\
					&\ll 
						V^6(Ux \log U \log x)^{1/2}\log^{32}x.
				\end{align*}}
				whenever
				\begin{equation*}
					\frac{x}{\log^Ax}\leq \modval V^4U\log U\leq x\qquad \text{ for some }A>0.\label{eq:davcond2}
				\end{equation*}

				Using the fact that
				$$\phi([\modval,mf^2,ng^2])=\frac{\phi(\modval)\phi([mf^2,ng^2])}{\phi(\modval,[mf^2,ng^2])}$$
				and the bound in \cref{lem:cstar2bound}, all of the analysis in \cite[Equations (19-26)]{adj} still holds.
				We have
				\begin{align*}
					\sum_{\substack{f,g\leq 2\sqrt{x}\\(fg,2\tau)=1}}&\frac{1}{fg}\sum_{\substack{p\in S_{f,g}(x)\\p\equiv \congval \pmod \modval}}L(1,\chi_{d_1})L(1,\chi_{d_2})\log p\\
					&=\frac{x}{\phi(\modval)}\sum_{f,g,m,n\geq 1}\frac{c_{f,g}^{\tau,\congval,\modval}(m,n)\phi(\modval,[mf^2,ng^2])}{fgmn \phi([mf^2,ng^2])}
						+O\Bigg(
							\frac{x^{11/8+2\ep}}{U^{1/2}}+U^{\delta}\log U\log^2 V\\
					&\qquad+\frac{x\log x\log^2 U}{V}+\frac{x\log x \log V^2}{(U\log U)^{1-\ep}}+U\log^2 U\log^2 V\\
					&\qquad+V^6(Ux \log U \log x)^{1/2}\log^{32}x+\frac{x}{(U\log U)^{1/2-\ep}}+\frac{x}{V^{1-2\ep}}+\frac{x}{U^{1/4}}\Bigg)
				\end{align*}
				whenever the condition \eqref{eq:davcond2} is satisfied.
				Furthermore, the infinite series converges absolutely.
				Once again note the weaker bound $O\la(\frac{x}{V^{1-2\ep}}\ra)$ that accounts for the terms $f\leq V<g$ and $g\leq V<f$ not considered in the analysis of \cite{adj}
				
				Let $U=\frac{x}{\modval\log^\alpha x}$ and $V=\log^\beta x$ for positive integers $\alpha,\beta$ such that $\alpha-4\beta-1\geq 1$, ensuring that the condition \eqref{eq:davcond2} is satisfied.
				Then for any $c>0$ we have 
				\begin{align*}
					\sum_{\substack{f,g\leq 2\sqrt{x}\\(fg,2\tau)=1}}&\frac{1}{fg}\sum_{\substack{p\in S_{f,g}(x)\\p\equiv \congval\pmod \modval}}L(1,\chi_{d_1})L(1,\chi_{d_2})\log p\\
					&=K_{\tau,\congval,\modval} x+O\la(\frac{x}{\log^{c}x}\ra)
				\end{align*}
				for a sufficiently large $\alpha,\beta$, where
				\begin{align*}
					K_{\tau,\tau,\congval,\modval} &=\frac{1}{\phi(\modval)}\sum_{f,g,m,n\geq 1}\frac{c_{f,g}^{\tau,\congval,\modval}(m,n)\phi(\modval,[mf^2,ng^2])}{fgmn \phi([mf^2,ng^2])}
				\end{align*}
				
				By the same analysis as in \cite{adj} we have
				\begin{align*}
					\sum_{\substack{p\leq x\\p\in\cP}}\frac{\H{\tau}{p}^2}{p^2}
					&=\frac{4}{\pi^2}\sum_{\substack{f,g\leq 2\sqrt x\\(fg,2\tau)=1}}\frac{1}{fg}\sum_{\substack{p\in S_{f,g}(x)\\p\equiv \congval\pmod \modval}}\frac{L(1,\chi_{d_1})L(1,\chi_{d_2})}{p}\\
					&=\frac{4K_{\tau,\tau,\congval,\modval}}{\pi^2} \log \log x+O(1).
				\end{align*}

				Finally, by \cref{lem:cstar2} we have
				\begin{align*}
					\sum_{i,j,k,\ell\geq 0}\frac{c_{p^k,p^\ell}^{\tau,\congval,\modval}(p^i,p^j)}{p^{i+j+k+\ell}\phi(p^{\max(v_p(\modval),i+2k,j+2\ell)})}	&=\Lambda_{\tau,\tau,\congval,\modval}(p),
				\end{align*}
				so by \cref{lem:cstar2mult} and \cref{lem:euler} we have
				\begin{align*}
					K_{\tau,\tau,\congval,\modval} 	&=\sum_{f,g,m,n\geq 1}\frac{c_{f,g}^{\tau,\congval,\modval}(m,n)}{fgmn \phi([\modval,mf^2,ng^2])}\\
										&=\prod_p\sum_{i,j,k,\ell\geq 0}\frac{c_{p^k,p^\ell}^{\tau,\congval,\modval}(p^i,p^j)}{p^{i+j+k+\ell}\phi(p^{\max(v_p(\modval),i+2k,j+2\ell)})}\\
										&=\prod_p \Lambda_{\tau,\tau,\congval,\modval}(p).\qedhere
				\end{align*}
			\end{proof}
		\end{prop}

	\subsection{One-Curve Hurwitz Number Sum}

	    The one-curve case was handled in \cite{jcong}, we re-express the result below.
		\begin{prop}
			For odd $\tau$, $(\congval,\modval)=1$, and all $c>0$ we have
			$$\sum_{\substack{p\leq x\\p\equiv \congval\pmod \modval}}\frac{\H{\tau}{p}}{p}=C_{\tau,\congval,\modval} \pi_{1/2}(x)+O\la(\frac{x}{\log^c x}\ra)$$
			where
			\begin{align*}
				C_{\tau,\congval,\modval}			&=\frac{4}{\pi}\prod_p\Lambda_{\tau,\congval,\modval}(p)>0,\\
				\Lambda_{\tau,\congval,\modval}(p)		&=	\pw{
													\frac{2}{3}																						&v_p(\modval)=0,p=2,\\
													\frac{p^2}{p^2-1}																					&v_p(\modval)=0,p\mid \tau,\\
													\frac{p(p^2-p-1)}{(p^2-1)(p-1)}																			&v_p(\modval)=0,p\nmid 2\tau,\\
													\frac{\sigma_{-1}\la(p^{\ceil{v_p(\modval)/2}-1}\ra)}{\phi(p^{v_p(\modval)})}+\frac{1}{p^{3\ceil{v_p(\modval)/2}-3}(p^2-1)(p-1)}			&1\leq v_p(\modval)\leq v_p(\rho_0),\\
													\frac{p}{\phi(p^{v_p(\modval)})\la[p-\kro{\rho_*}{p}\ra]}															&0=v_p(\rho_0)<v_p(\modval),\\
													\frac{\sigma_{-1}\la(p^{v_p(\rho_0)/2-1/2}\ra)}{\phi(p^{\max(v_p(\modval))})}											&0<v_p(\rho_0)<v_p(\modval),2\nmid v_p(\rho_0),\\
													\frac{1}{\phi(p^{v_p(\modval)})}\la[\sigma_{-1}\la(p^{v_p(\rho_0)/2}\ra)+\frac{1}{p^{v_p(\rho_0)/2}\la[\kro{\rho_*}{p}p-1\ra]}\ra]			&0<v_p(\rho_0)<v_p(\modval),2\mid v_p(\rho_0).
												}	
			\end{align*}
			with $\rho_0=\tau^2-4\congval$, $\rho_*=p^{-v_p(\tau^2-4\congval)}(\tau^2-4\congval)$, and $\sigma_\alpha$ is the sum-of-divisors function $\sigma_\alpha(n)=\sum_{d\mid n}d^\alpha$.
			\end{prop}

			Note that both constants $C_{\tau,\congval,\modval}$ and $C_{\tau,\tau,\congval,\modval}$ depend significantly on the congruence class $\congval\mod\modval$, similar to what was found in \cite{w} for the $\tau=0$ case.
			This means that on average the set $\set{p:a_p(E)=\tau}$ is biased towards certain congruence classes.

\section{Example Argument Sets
\label{sec:ex}}	
	\subsection{Integers}

		Let 
		\begin{align*}
			\I(T)			&=\set{1,\dots,T}.
		\end{align*}
		We have
		\spanner{\begin{align*}
			\#\I(T)		    &=T,\\
			R_{\I(T),p}(w)	&=Tp^{-1}+O(1).
		\end{align*}}
		For $\S(T)=\I(T)$ we have the following corollaries of the previous sections.
		
		\begin{cor}\label{cor:I}
			For all $0\leq \delta\leq 1/2$ we have
			\begin{align*}
				\sum_{\substack{t\in \I(T)\\\Delta(t)\neq 0}}\pit{t}(\A;x)
						&\ll	\pw{
									T\frac{x^{1/2}}{\log x}+\frac{x^{3/2}}{\log x}			&		\A=\tau,\\
									T\frac{x^{1/4}}{\log x}+\frac{x^{5/4}}{\log x}			&		\A(p)=\pm \floor{2\sqrt{p}},\\
									T x^{\delta/2+1/4}\log_2  x+x^{\delta/2+5/4}\log_2  x		&		\A(p)=\pm 2\sqrt{p}+O\la(p^{\delta}\ra),\\
									T x^{1/2}\log_2  x+x^{3/2}\log_2  x					&		\oth.
							}
			\end{align*}
		\end{cor}
			In the general case \cite[Theorem 7]{ss} states
			$$\sum_{\substack{t\in \I(T)\\\Delta(t)\neq 0}}\pit{t}(\A;x)\ll T^2+T^{1/2}x^{5/4+o(1)}.$$
			We also have the trivial bound
			$$\sum_{\substack{t\in \I(T)\\\Delta(t)\neq 0}}\pit{t}(\A;x)\ll Tx.$$
			\cref{cor:I} is non-trivial and an improvement over \cite[Theorem 7]{ss} for $T\gg x^{1/2+\ep}$, $\ep>0$.
			For $T\ll x^{1/2}$ the trivial bound is better than \cref{cor:I} and \cite[Theorem 7]{ss}.

			For the extremal $\A$ case \cref{cor:I} is non-trivial and an improvement over \cite[Theorem 7]{ss} for $T\gg x^{1/4+\ep}$ for $\ep>0$.
			
		\begin{cor}\label{cor:I2}
			For all $0\leq \delta_1,\delta_2<1/2$ we have
			\begin{align*}
				&\sum_{\substack{t_1,t_2\in \I(T)\\\Delta_1(t_1)\neq 0\neq \Delta_2(t_2)}}\pitt{t_1}{t_2}(\A_1,\A_2;x)\\
				&\qquad\ll \pw{
							T^2\log_2  x+Tx+x^2															& 	\A_i=\tau_i,\\	
							T^2+Tx^{1/2}\log_2  x+x^{3/2}\log_2  x												&	\A_i(p)=\pm \floor{2\sqrt p},\\	
							 T^2+Tx^{\frac{1+\delta_1+\delta_2}{2}}\log x\log_2^2 x+x^{\frac{3+\delta_1+\delta_2}{2}}\log x\log_2^2 x	&	\A_i(p)=\pm 2\sqrt{p}+O\la(p^{\delta_i}\ra),\\
							 T^2\log^2 x\log_2^3 x+Tx \log x\log_2^2 x+x^2\log x\log_2^2 x								&	\oth.
						}
			\end{align*}
		\end{cor}
			Also note that $W(T)=1+\log \max\I(T)=1+\log T$ so by \cref{lem:isogenous} the sum over pairs of isogenous curves is insignificant compared to the bounds in \cref{cor:I2}.

	\subsection{Farey Fractions}

		Let 
		\begin{align*}
			\cF(T)			&=\set{u/v\in \Q:(u,v)=1,1\leq u,v\leq T}.
		\end{align*}
		To bound $R_{\cF(T),p}(w)$ we need the following.
		\begin{lem}\label{lem:copran}
			For any $v\in \N$ and real $t_0<t_1$ we have
			\begin{align*}
				\sum_{\substack{t_0\leq k\leq t_1\\(v,k)=1}}1	&=\frac{t_1-t_0+1}{v}\phi(v)+O(d(v))
			\end{align*}
			\begin{proof}
				We have
				\begin{equation*}
					\sum_{\substack{t_0\leq k\leq t_1\\(v,k)=1}}1
													=\sum_{\substack{t_0\leq k\leq t_1\\b\mid(v,k)}}\mu(b)
													=\sum_{b\mid v}\mu(b)\la(\frac{t_1-t_0+1}{b}+O(1)\ra)
													=\frac{t_1-t_0+1}{v}\phi(v)+O(d(v)).\qedhere
				\end{equation*}
			\end{proof}
		\end{lem}
		Note that $\sum_{v\leq T}d(v)\ll T\log T$ and by \cite{mont}
		$$\sum_{v\leq T}\frac{\phi(v)}{v}=\frac{6}{\pi^2}T+O(\log^{2/3}T\log_2^{4/3} T).$$
		Therefore
		\begin{align*}
			\#\cF(T)		=\sum_{1\leq v\leq T}\sum_{\substack{1\leq u\leq T\\(u,v)=1}}1
						=\sum_{1\leq v\leq T}\la[\frac{T\phi(v)}{v}+O(d(v))\ra]
						=\frac{6}{\pi^2}T^2+O(T\log T),
		\end{align*}
		\begin{align*}
			R_{\cF(T),p}(w)	=\sum_{\substack{1\leq u,v\leq T\\(u,v)=1\\p\nmid v\\u\equiv wv\pmod p}}1					
						=\sum_{\substack{1\leq v\leq T\\p\nmid v}}\sum_{\substack{\frac{1-wv}{p}\leq k\leq \frac{T-wv}{p}\\(v,k)=1}}1
						&=\sum_{\substack{1\leq v\leq T\\p\nmid v}}\la[\frac{T\phi(v)}{pv}+O(d(v))\ra]\\
						&=\frac{6}{\pi^2}T^2p^{-1}+O\la(T\log T+T^2p^{-2}\ra)\\
						&=\#\cF(T)p^{-1}+O\la(T\log T+\#\cF(T)p^{-2}\ra).
		\end{align*}
		For the bounds in Section \ref{sec:basic} we get better results using
		\begin{align*}
			\#\cF(T)		&\ll T^2,\\
			R_{\cF(T),p}(w)	&=\sum_{\substack{1\leq v\leq T\\p\nmid v}}\sum_{\substack{1\leq u\leq T\\(u,v)=1\\u\equiv vw\pmod p}}1
						\leq \sum_{1\leq v\leq T}\la[Tp^{-1}+1\ra]
						= \frac{T^2}{p}+T.
		\end{align*}
		\begin{cor}\label{cor:F}
			For all $0\leq \delta_1,\delta_2\leq 1/2$ we have
			\begin{align*}
				\sum_{\substack{t\in \cF(T)\\\Delta(t)\neq 0}}\pit{t}(\A;x)
						&\ll	\pw{
									T^2\frac{x^{1/2}}{\log x}+T\frac{x^{3/2}}{\log x}			&		\A=\tau,\\
									T^2 \frac{x^{1/4}}{\log x}+T \frac{x^{5/4}}{\log x}		&		\A(p)=\pm \floor{2\sqrt{p}},\\
									T^2 x^{1/4+\delta/2}\log_2  x+T x^{5/4+\delta/2}\log_2  x	&		\A(p)=\pm 2\sqrt{p}+O\la(p^{\delta}\ra),\\
									T^2  x^{1/2}\log_2  x+T x^{3/2}\log_2  x				&		\oth.
							}
			\end{align*}
		\end{cor}
		\begin{cor}\label{cor:ssFLin}
			If $\deg j=1$ then
			$$\sum_{\substack{t\in \cF(T)\\\Delta(t)\neq 0}}\pit{t}(\tau;x)=\frac{6}{\pi^2}C_{\tau}T^2\pi_{1/2}(x)+O\la( T^2\frac{x^{1/2}}{\log^c x}+T\log T\frac{x^{3/2}}{\log x}\ra).$$
		\end{cor}
	
			\cite[Theorem 1]{ss} states
			$$\sum_{\substack{t\in \cF(T)\\\Delta(t)\neq 0}}\pit{t}(\A;x)\ll T^2x^{7/8}+Tx^{11/8+o(1)}.$$
			\cref{cor:F} is non-trivial and an improvement over both \cite[Theorem 1]{ss} for $T\gg x^{5/8+\ep}$, $\ep>0$.
			However, for $x^{3/8+\ep}\ll T\ll x^{5/8-\ep}$, $0<\ep<1/8$, best bound is \cite[Theorem 1]{ss} and for $T\ll x^{3/8}$ the trivial bound is best.

			For constant $\A=\tau$ case \cite[Equation (11)]{ss} gives the bound $T^2x^{3/4}+Tx^{3/2+o(1)}$
			\cref{cor:F} is non-trivial and an improvement over both \cite[Theorem 1]{ss} and \cite[Equation (11)]{ss} for $T\gg x^{3/4+\ep}$, $\ep>0$.
			For $x^{5/8+\ep}\ll T\ll x^{3/4-\ep}$ the best bound is \cite[Equation (11)]{ss}.
			For $x^{3/8+\ep}\ll T\ll x^{5/8-\ep}$ the best bound is \cite[Theorem 1]{ss}.
			For $T\ll x^{3/8}$ the trivial bound is best.

			For the $\A=0$ case \cite[Theorem 1]{ss} gives the bound $T^2x^{5/6}+Tx^{4/3+o(1)}$
			and \cite[Equation (11)]{ss} gives the bound $T^2x^{2/3}+Tx^{3/2+o(1)}$.
			\cref{cor:F} is non-trivial and an improvement over both \cite[Theorem 1]{ss} and \cite[Equation (11)]{ss} for $T\gg x^{5/6+\ep}$, $\ep>0$.
			For $x^{2/3+\ep}\ll T\ll x^{5/6-\ep}$ the best bound is \cite[Equation (11)]{ss}.
			For $x^{1/3+\ep}\ll T\ll x^{2/3-\ep}$ the best bound is \cite[Theorem 1]{ss}.
			For $T\ll x^{1/3}$ the trivial bound is best.

			For $j(Z)=Z$, \cite[Theorem 2]{ss} gives the bound $T^2x^{3/4+o(1)}+Tx^{5/4+o(1)}$.
			\cref{cor:F} is non-trivial and an improvement over both \cite[Theorem 2]{ss} for $T\gg x^{3/4+\ep}$, $\ep>0$.
			For $x^{1/4+\ep}\ll T\ll x^{3/4-\ep}$ the best bound is \cite[Theorem 1]{ss}.
			For $T\ll x^{1/4}$ the trivial bound is best.

			For the extremal $\A$ case \cref{cor:F} is non-trivial and an improvement over \cite[Theorem 1]{ss} for $T\gg x^{1/4+\ep}$ for $\ep>0$.

	\subsection{Multiset of Sums of Integers}
		\newcommand{\multiplus}{+}
		Note that all of the above propositions are compatible with $\S(T)$ being a finite multiset.
		More specifically, for any multiset $\mathcal A$ with finite support set $A$ and multiplicity function $m_{\mathcal A}:A\to \N$, and any function $f:A\to \R$ we define
		$$\sum_{a\in \mathcal A}f(a):=\sum_{a\in A}m_{\mathcal A}(a)f(a).$$
		Its cardinality is defined as
		$$\#\mathcal A:=\sum_{a\in A}m_{\mathcal A}(a).$$
		For any sets $A,B\subset \Q$, define the multi-sumset $A\multiplus B$ with support the standard sumset $\set{a+b:a\in A,b\in B}$ and multiplicity function given by
		$$m_{A\multiplus B}(c)=\#\set{(a,b)\in A\times B:a+b=c}.$$
		The above definitions are set up so that
		\begin{align*}
			\sum_{c\in A\multiplus B}f(c)	&=\sum_{\substack{a\in A\\b\in B}}f(a+b).
		\end{align*}

		\begin{cor}\label{cor:ssUVImp}	
			Let $\U(T),\V(T)\subset \I(T)$.		
			If
			\begin{align*}
				R_{\U(T),p}(w)	&=\#\U(T)p^{-1}+O(1),\\
				R_{\V(T),p}(w)	&=\#\V(T)p^{-1}+O(1)
			\end{align*}
			then
			$$\sum_{\substack{t\in \U(T)\multiplus \V(T)\\\Delta(t)\neq 0}}\pit{t}(\A;x)\ll \#\U(T)\#\V(T)x^{1/2}\log_2  x+\min(\#\U(T),\#\V(T))x^{3/2}\log_2  x.$$
			\begin{proof}
				We have $\#\U(T)\multiplus \V(T)=\#\U(T)\#\V(T)$ and
				\begin{align*}
					R_{\U(T)\multiplus\V(T),p}(w)		=\sum_{\substack{t\in \U(T)\multiplus\V(T)\\t\equiv w\pmod p}}1
										=\sum_{\substack{r\in \U(T)\\s\in \V(T)\\r+s\equiv w\pmod p}}1
										&=\sum_{r\in \U(T)}(\#\V(T)p^{-1}+O(1))\\
										&=\#\U(T)\#\V(T)p^{-1}+O(\#\U(T)),
				\end{align*}	
				and by symmetry $R_{\U(T)\multiplus\V(T),p}(w)=\#\U(T)\#\V(T)p^{-1}+O(\#\V(T))$ so
				$$R_{\U(T)\multiplus\V(T),p}(w)=\#\U(T)\#\V(T)p^{-1}+O(\min(\#\U(T),\#\V(T))).$$
				The result then follows from \cref{prop:set}(i).
			\end{proof}
		\end{cor}

			\cite[Theorem 8]{ss} states
			$$\sum_{\substack{t\in \U(T)\multiplus \V(T)\\\Delta(t)\neq 0}}\pit{t}(\A;x)\ll T\#\U(T)\#\V(T)+(\#\U(T)\#\V(T))^{3/4}x^{5/4}.$$
			Note however that \cite[Theorem 8]{ss} does not put any restrictions on $R_{\U(T),p}(w)$ and $R_{\V(T),p}(w)$ as we have done.
			\cref{cor:ssUVImp} is non-trivial and an improvement over \cite[Theorem 8]{ss} for $\max(\#\U(T),\#\V(T))\gg x^{1/2+\ep}$, $\ep>0$.
			For $\#\U(T),\#\V(T)\ll x^{1/2}$ the trivial bound is best.

	\subsection{Multiset of Sums of Farey Fractions}
		If $\frac{u}{v}=\frac{u_1}{v_1}+\frac{u_2}{v_2}$ where $(u,v)=(u_1,v_1)=(u_2,v_2)=1$ and $p\nmid v$ then either $p\nmid v_1v_2$ or $p\mid (v_1,v_2)$.
		Note that there are $O\la(T^4p^{-2}\ra)$ elements $\frac{u_1}{v_1}+\frac{u_2}{v_2}\in \cF(T)\multiplus \cF(T)$ (counted with multiplicity) such that $p\mid (v_1,v_2)$.
		We have $\#\cF(T)\multiplus \cF(T)\ll T^4$ and
		\begin{align*}
			R_{\cF(T)\multiplus \cF(T),p}(w)	=\sum_{\substack{\frac{u}{v}\in \cF(T)\multiplus \cF(T)\\p\nmid v\\u\equiv vw\pmod p}}1
									&\ll T^4p^{-2}+\sum_{\substack{\frac{u_1}{v_1},\frac{u_2}{v_2}\in \cF(T)\\p\nmid v_1v_2\\u_1v_1\inv+u_2v_2\inv\equiv w\pmod p}}1\\
									&\ll T^4p^{-2}+\sum_{\substack{\frac{u_1}{v_1}\in \cF(T)\\p\nmid v_1}}\la[T^2p^{-1}+T\ra]\\
									&\ll T^4p^{-1}+T^3.
		\end{align*}
		\begin{cor}\label{cor:FF}
			We have
			\begin{equation*}
				\sum_{\substack{t\in \cF(T)\multiplus \cF(T)\\\Delta(t)\neq 0}}\pit{t}(\A;x)	\ll 	\pw{
																		T^4\frac{x^{1/4}}{\log x}+T^3\frac{x^{5/4}}{\log x}			&\A(p)=\pm\floor{2\sqrt p},\\
																		T^4x^{1/2}\log_2  x+T^{3}x^{3/2}\log_2  x				& \oth.
																}
			\end{equation*}
		\end{cor}

			\cite[Theorem 3]{ss} states that
			$$\sum_{\substack{t\in \cF(T)\multiplus \cF(T)\\\Delta(t)\neq 0}}\pit{t}(\A;x)	\ll 	T^5+T^4x^{3/4+o(1)}+T^3x^{5/4+o(1)}.$$
			\cref{cor:FF} is non-trivial and an improvement over \cite[Theorem 3]{ss} for $T\gg x^{3/4+\ep}$, $\ep>0$.
			For $x^{1/4+\ep}\ll T\ll x^{3/4-\ep}$ the best bound is \cite[Theorem 3]{ss}.
			For $T\ll x^{1/4}$ the trivial bound is best.

			For the extremal $\A$ case \cref{cor:FF} is non-trivial and an improvement over \cite[Theorem 3]{ss} for $T\gg x^{1/2+\ep}$, $\ep>0$.
			For $x^{1/4+\ep}\ll T\ll x^{1/2-\ep}$ the best bound is \cite[Theorem 3]{ss}.
			For $T\ll x^{1/4}$ the trivial bound is best.

\end{document}